\documentclass{amsart}
\usepackage{amssymb,enumerate,bm,color,here,url,hyperref}

\usepackage{graphicx} 

\theoremstyle{plain}
	\newtheorem{thm}{Theorem}[section]
	\newtheorem{lem}[thm]{Lemma}
	\newtheorem{prop}[thm]{Proposition}

	\newtheorem{ex}[thm]{Example}
\theoremstyle{definition}
	\newtheorem{dfn}[thm]{Definition}
\theoremstyle{remark}
	\newtheorem{rem}[thm]{Remark}

\DeclareMathOperator{\Aut}{Aut}
\DeclareMathOperator{\Isom}{Isom}
\DeclareMathOperator{\lcm}{lcm}
\DeclareMathOperator{\Out}{Out}
\DeclareMathOperator{\vol}{vol}

\newcommand{\cut}{\setminus \! \! \setminus}

\newcommand{\bbZ}{\mathbb{Z}}
\newcommand{\bbQ}{\mathbb{Q}}
\newcommand{\bbR}{\mathbb{R}}

\begin{document}

\title{On the isotopies of tangles in periodic 3-manifolds using finite covers}

\author[Y. Kotorii]{Yuka Kotorii}
\address[Kotorii]{Mathematics Program, Graduate School of Advanced Science and Engineering, Hiroshima University, 1-7-1 Kagamiyama, Higashi-hiroshima, Hiroshima 739-8521 Japan}
\address[Kotorii, Mahmoudi, Matsumoto, Yoshida]{International Institute for Sustainability with Knotted Chiral Meta Matter (WPI-SKCM$^2$), Hiroshima University, 1-3-1 Kagamiyama, Higashi-Hiroshima, Hiroshima 739-8531, Japan}
\address[Kotorii, Mahmoudi]{RIKEN, interdisciplinary Theoretical and Mathematical Sciences Program, 2-1, Hirosawa, Wako, Saitama 351-0198, Japan}
\email{kotorii@hiroshima-u.ac.jp}

\author[S. Mahmoudi]{Sonia Mahmoudi}
\address[Mahmoudi]{Advanced Institute for Materials Research (WPI-AIMR), Tohoku University, 2-1-1 Katahira, Aoba-ku, Sendai, Miyagi 980-8577, Japan}
\email{sonia.mahmoudi@tohoku.ac.jp}

\author[E. A. Matsumoto]{Elisabetta A. Matsumoto}
\address[Matsumoto]{School of Physics, Georgia Institute of Technology, 837 State Street, Atlanta, GA 30332, U.S.A.}
\email{sabetta@gatech.edu}

\author[K. Yoshida]{Ken'ichi Yoshida}
\email{kncysd@hiroshima-u.ac.jp}

\date{}

\begin{abstract}
A periodic tangle is a one-dimensional submanifold in $\bbR^3$ that has translational symmetry in one, two, or three transverse directions. 
A periodic tangle can be seen as the universal cover of a link in the solid torus, the thickened torus, or the three-torus, respectively.
Our goal is to study equivalence relations of such periodic tangles.
Since all finite covers of a link lift to the same periodic tangle, it is necessary to prove that isotopies between different finite covers are preserved.
In this paper, we show that if two links have isotopic lifts in a common finite cover, then they are isotopic. 
To do so, we employ techniques from 3-manifold topology to study the complements of such links.
\end{abstract}

\maketitle

\section{Introduction}
\label{section:intro}

\emph{Periodic tangles} are complex entanglements of curves embedded in $\bbR^{3}$ that are periodically repeated in either one, two, or three directions. A periodic tangle can be described as the lift of a link in certain 3-manifolds to its universal cover. More specifically, a \emph{singly periodic \textup{(}SP\textup{)} tangle} can be seen as the lift of a link in the solid torus $S^{1} \times D^{2}$ into $\bbR \times D^{2} \subset \bbR^3$, as described in \cite{SZ25}. A \emph{doubly periodic \textup{(}DP\textup{)} tangle} can be defined as the lift of a link in the thickened torus $T^{2} \times I$ into $\bbR^{2} \times I$, where $I$ is a closed interval, as done in \cite{DLM26}. Finally, a \emph{triply periodic \textup{(}TP\textup{)} tangle} can be described as the lift of a link in the 3-torus $T^{3}$ into $\bbR^{3}$, as presented in \cite{AEM25}. 

Periodic tangles have been used to model a variety of entangled materials. In chemistry, such systems of entangled chains of filaments or catenanes appear in polymers \cite{Panagiotou15, PM18, PML13, Treacy2}, crystalline materials \cite{Treacy1, Evans}, and the mitochondrial DNA of certain protozoa \cite{Shapiro95, KSD20}. Periodic tangles also describe macroscopic materials, such as knitted and woven textiles \cite{GMO09a, GMO09b, MG09,MM20,FKM23,FKM26,Mahmoudi20}, whose physical properties can often be linked to their topology \cite{LOTY18,singal2024}. In all of these systems, the bulk material closely resembles the universal cover of some link in the solid torus, the thickened torus, or the 3-torus. Instead of dealing directly with periodic tangles in the universal cover $\bbR^{3}$, many researchers consider more minimal representations using links in 3-manifolds, called \emph{motifs}, that lift to the periodic tangle in question \cite{GMO09a, GMO09b, MG09,MM20}. Textiles, in particular, are represented by links in the thickened torus. When describing the same universal cover, two different researchers might choose different motifs to study. This can be practical from a scientific point of view, but it may prove to be confusing when trying to classify the topology of such tangles. For instance, some topological invariants may evaluate to different values, despite referring to the same system in the universal cover.

\smallbreak

The topological \emph{equivalence} of periodic tangles was first investigated by Grishanov, Meshkov, and Omelchenko in \cite{GMO07}, whose prime motivation was the study and classification of doubly periodic textile structures through a diagrammatic theory on an integer lattice
\footnote{In particular, one may consider a basis $B$ of $\mathbb{E}^2$ and a periodic lattice $\Lambda$ isomorphic to $\mathbb{Z}^2$ associated to a doubly periodic tangle diagram $D$. Thus, the quotient of $D$ under $\Lambda$ generates a motif diagram of a doubly periodic tangle. Such a motif is not unique, as infinitely many different point lattices or different choices of basis for the same lattice may be chosen to quotient $D$.}.
This approach was generalized to the study of doubly periodic tangles in \cite{DLM26}, where global isotopies are discretized on the diagrammatic level as sequences of moves on the corresponding motif diagrams. These include three types of isotopies: 
\begin{enumerate}
\item acting on the local level, 
\item the choice of a basis of the universal cover, and 
\item the number of periodic units in a fundamental domain. 
\end{enumerate}

These latter two notions, which extend to both singly periodic tangles \cite{SZ25} and triply periodic tangles \cite{AEM25}, highlight the main difference between isotopies of classical links embedded in a 3-manifold and those of periodic tangles. 
Note that all of these approaches consider a diagrammatic theory of periodic tangles, which has led to the construction of topological invariants such as several numerical, polynomial, and finite type invariants \cite{AEM25, BK20, DLM24, DLM26,  GMO07, GMO09a, GMO09b, GMV09, GV11, Kawauchi18, MG09, SZ25}.  

\smallbreak

Another interesting approach to understanding the topology of periodic tangles is to study the complement of links embedded in $S^{1} \times D^{2}$, $T^{2} \times I$, and $T^3$ to account for periodicity.
Given a link $L$ in a 3-manifold $X$, removing a tubular neighborhood of $L$ from $X$ results in a 3-manifold known as the \emph{link complement}, denoted $X \cut L$. A link complement can be systematically constructed from a diagram. Software such as \emph{SnapPy} does this by constructing a 3-manifold complement combinatorially, starting with the crossings in a diagram \cite{SnapPy}. As a result of the geometrization theorem, such 3-manifold complements can (after cutting along appropriate spheres and tori) be given metrics consistent with one of eight uniform geometric structures \cite{Thurston82, Perelman02, Perelman03a, Perelman03b}. It is important to note that most knot and link complements are equipped with \emph{hyperbolic structures}, which have constant sectional curvature $-1$. 

While a knot is completely defined by its complement \cite{GordonLuecke}, only the majority of the topology of a link is captured by its complement. (We will address the ambiguity arising here in our discussion of Dehn twists and covers of motifs.)
The properties of the complement 3-manifold can be used to characterize the topology of the knot or link.
These are particularly nice when the knot or link complement is a hyperbolic manifold, as invariants of the hyperbolic complement 3-manifold are also invariants of the knot or link.
The hyperbolic volume is a topological invariant of a knot that is not directly captured by knot diagrammatic techniques.
For example, two-periodic weft-knitted textiles were investigated using hyperbolic structures in \cite{KMMP25}. 
In addition, interesting examples of hyperbolic 3-manifolds have been generated by links in the thickened torus and the 3-torus, as investigated in several studies \cite{Adams20, CKP16, CKP19, Hui25, HP24, Kwon25, KPT23, KT22, LAT04, Yoshida18, Yoshida22}.

As with the diagrammatic picture, the three levels of isotopy that act on a DP tangle $L$ exist when studying the 3-manifold complement $T^2\times I \cut L$. 
\begin{enumerate}
\item \emph{Local isotopies} are realized by ambient isotopies of the manifold $T^2\times I \cut L$. 
\item To change the basis set of the universal cover, let us consider the structure of $T^2\times I$. The manifold $T^2\times I$ can be obtained by drilling the tubular neighborhood around a Hopf link in $S^3.$ The boundary generated by this surgery are identified with $T^2\times \{0\}$ and $T^2\times \{1\}$. A Dehn twist in the diagrammatic picture acts on the components of the Hopf link. This would change the frame on $T^2\times \{0\}$ and/or $T^2\times \{1\}$. However, that does not change the structure of the manifold $T^2\times I$. (This is the information we have lost between Dehn twists in the diagrammatic picture and the 3-manifold picture.) Thus, we find that isotopies between different basis choices on the universal cover are trivial in the 3-manifold picture. 
\item This paper demonstrates that \emph{scale equivalence} is an equivalence relation between finite covers. We will also show that this equivalence relation holds for singly periodic and triply periodic tangles.
\end{enumerate}

\subsection{Summary of results}
\label{subsection:summary}

In this paper, we aim to explore the scale equivalence relation of singly periodic, doubly periodic, and triply periodic tangles. We claim that if two links in $S^{1} \times D^{2}$, $T^{2} \times I$, or $T^{3}$ are isotopic, then their finite covers must also be isotopic. 
To do so, we introduce the notion of \emph{cover-isotopy} of links, which defines the existence of a finite covering space of the manifold such that the links, when lifted to this covering space, become isotopic. 
This leads to the following main theorems, whose proofs are presented in Section~\ref{section:isotopy}:

\begin{thm}[Theorems \ref{thm:1-isotopy}, \ref{thm:2-isotopy}, and \ref{thm:3-isotopy}]
\label{thm:main1}
Suppose that $X$ is the solid torus $S^{1} \times D^{2}$, the thickened torus $T^{2} \times I$, or the 3-torus $T^{3}$. 
Let $L_{0}$ and $L_{1}$ be links in $X$. 
Let $P \colon \widetilde{X} \to X$ be a finite covering map. 
Suppose that the lifts $\widetilde{L_{0}} = P^{-1}(L_{0})$ and $\widetilde{L_{1}} = P^{-1}(L_{1})$ are isotopic in $\widetilde{X}$. 
Then $L_{0}$ and $L_{1}$ are isotopic in $X$. 
\end{thm}

The statement of Theorem~\ref{thm:main1} is surprisingly non-trivial. On one hand, it seems fairly obvious that if two links $\widetilde{L_0}$ and $\widetilde{L_1}$ in $\widetilde{X}$ are isotopic, their images $L_0= P(\widetilde{L_0})$ and $L_1= P(\widetilde{L_1})$ should be isotopic. However, going the other way introduces several complexities. For example, studies in other three-manifolds, including lens spaces and the thickened Klein bottle, have shown that cover-isotopy does not imply isotopy in general \cite{Manfredi14,Yoshida25a}.

Two braids $\alpha$ and $\beta$ have isotopic closures in the solid torus if and only if $\alpha$ and $\beta$ are conjugate in the braid group $B_{k}$ \cite[Theorem 2.1]{KT08}. 
Hence Theorem~\ref{thm:main1} implies the following result of Gonzalez-Meneses \cite{Gonzalez03}: 
If $\alpha$ and $\beta$ in $B_{k}$ satisfy that $\alpha^{n} = \beta^{n}$ for some $n > 0$, then $\alpha$ and $\beta$ are conjugate. 
The proof in \cite{Gonzalez03} relies on the Nielsen-Thurston classification of braids, 
and so it seems similar to the proof of Theorem~\ref{thm:main1}. 

\smallbreak

In the theory of periodic tangles, one of the most challenging problems is the identification of a \emph{minimal motif}, as discussed in \cite{GMO07} for the case of doubly periodic tangles. A \emph{motif} is the quotient of a periodic tangle under the generators of its translational symmetries. A \emph{minimal motif} is a motif that cannot be a finite cover of another motif. 
A link is called \emph{split} if it has a 2-sphere in its complement separating one or more link components from the others. 
While this is defined for links in our three periodic 3-manifolds, we extend this definition to tangles in the universal cover of these manifolds. 
If a link $L$ in $S^{1} \times D^{2}$, $T^{2} \times I$, and $T^3$ is split, we say that the corresponding periodic tangle $L_\infty$ is also split. 
Note that doubly periodic tangles arising from typical textile structures are not split. 
In Section~\ref{section:minimal}, we prove the following theorem:

\begin{thm}[Theorem \ref{thm:minimal-motif}]
\label{thm:main2}
Let $L_{\infty}$ be a non-split periodic tangle with a motif $L \subset X = S^{1} \times D^{2}$, $T^{2} \times I$, or $T^3$. 
Suppose that each Seifert fibered JSJ piece $M$ of the complement of $L$ such that the homomorphism $\iota_{*} \colon H_{1}(M, \bbZ) \to H_{1}(X, \bbZ)$ induced by the inclusion map is non-trivial can be re-embedded as the complement of $T_{0}(p,0)$, $T_{1}(0,q)$, $T_{2}(p,q)$, or $T_{3}(p,q,r)$. 
Then $L_{\infty}$ has a motif $L_{\min}$ such that for any motif $L'$ of $L_{\infty}$, there exists a finite covering map $L' \to L_{\min}$. 
In particular, $L_{\infty}$ has a minimal motif unique up to homeomorphism. 
\end{thm}

The notation for Seifert JSJ pieces and re-embedding will be explained in Section~\ref{section:jsj}. 
The assumption of Theorem~\ref{thm:main2} holds in the case that $L_{\infty}$ is hyperbolic (i.e. it has a motif whose complement admits a hyperbolic structure). 
By Theorem~\ref{thm:main2}, it is sufficient to consider isotopy of minimal motifs for the classification of some types of periodic tangles. 
However, it is not easy in general to confirm the minimality of a motif. 
Nonetheless, the symmetry of a hyperbolic periodic tangle can be computed by SnapPy \cite{SnapPy}, unless it is too complicated. 

To prove our main theorems, we use the standard decomposition of the complement of a link: the prime decomposition and the JSJ decomposition. Each piece after the decomposition is Seifert fibered or hyperbolic. We will classify the types of JSJ tori and the Seifert fibered pieces, and apply the Mostow rigidity theorem to the hyperbolic pieces. 
Then we will use the result of Budney \cite{Budney06} for the JSJ decomposition of the complement of a link in $S^{3}$. 
Because the JSJ decomposition often induces links in $X=S^{1} \times D^{2}$, 
we shall use these results to prove cover-isotopy for $X = T^{2} \times I$ or $T^{3}$ in full generality. 
We believe that this novel approach will help develop new invariants of periodic tangles and find applications in other fields, including textiles, polymers, molecular chemistry, and k-DNA, amongst others.

This paper is organized as follows: 
first, in Section~\ref{section:prelim}, we recall essential notions for 3-manifolds and links. 
Next, in Section~\ref{section:jsj}, we describe the JSJ decomposition of the complements of links in the solid torus, the thickened torus, and the 3-torus. Then, in Section~\ref{section:isotopy}, we consider the cover-isotopy of links, and prove Theorem~\ref{thm:main1}. Finally, in Section~\ref{section:minimal}, we consider minimal motifs of periodic tangles, and prove Theorem~\ref{thm:main2}.

\section{Preliminaries}
\label{section:prelim}

In this section, we recall the definitions of links in 3-manifolds and their isotopies. 
We also recall the notions of prime decomposition, JSJ decomposition, Seifert fibrations, and hyperbolic 3-manifolds that are used to prove our main results. 
Note that within this paper, every 3-manifold that we consider is assumed to be connected, orientable, compact, and possibly possesses boundary unless otherwise stated.

\subsection{Links in 3-manifolds, periodic tangles, and minimal motifs}

Let $X$ be a connected 3-manifold, possibly with boundary. 
A \emph{link} $L$ in $X$ is a closed 1-submanifold embedded in the interior of $X$. 
We consider non-oriented, smooth (or piecewise linear) links. 
Two links $L_{0}$ and $L_{1}$ in $X$ are said to be \emph{isotopic} if there exists an \emph{ambient isotopy} between $L_{0}$ and $L_{1}$. 
This means that there exists a continuous map $F \colon X \times [0,1] \to X$ such that $F(\cdot, t) \colon X \to X$ is a homeomorphism for every $t$, where $F(\cdot, 0)$ is the identity map, and $F(L_{0}, 1) = L_{1}$. 
Thus, isotopy is an equivalence relation for links in a connected 3-manifold. 

In this paper, we consider links in the following 3-manifolds: $X = S^{1} \times D^{2}$, $T^{2} \times I$, or $T^{3}$, where $I=[0,1]$ is a closed interval.
We are interested in studying the lift of these links to the universal cover of $X$, that is, $\bbR \times D^{2}$, $\bbR^{2} \times I$, or $\bbR^3$. In the universal cover, these links are lifted to tangles with translational periodicity, denoted $L_\infty$ and called \emph{periodic tangles}.
Conversely, consider the map $P \colon \widetilde{X} \rightarrow X$ to be a finite covering map. 
Although $\widetilde{X}$ is homeomorphic to $X$, 
we distinguish $\widetilde{X}$ from $X$ to avoid confusion.
This map can be extended to define finite covers of links in $X$. We say that $\widetilde{L}$ is a \emph{finite cover} of $L$, if there exists a map $P \colon (\widetilde{X}, \widetilde{L})\rightarrow (X,L)$, where $P$ is a finite covering map in the usual sense and $\widetilde{L} = P^{-1}(L)$.

\smallbreak

We now define the notion of equivalence of periodic tangles as follows:

\begin{dfn}(\cite{DLM26})
\label{def:equivalence}
Two periodic tangles are said to be \emph{equivalent} if they can be obtained from each other by periodic isotopies and invertible affine transformations of the space carrying along the periodic tangles that are isotopic to the identity.
\end{dfn}

\begin{rem}
Periodic isotopies are isotopies in the usual sense that preserve periodicity in the universal cover of $X$. Thus it also applies to any finite covering map. Non-periodic local isotopies are not considered in this study.
Note that periodic lattices are not fixed in this definition. 
In other words, it is sufficient to preserve a large lattice. 
\end{rem}

Likewise, to see how isotopies of periodic tangles in the universal cover $X_\infty$ affect the corresponding links in our 3-manifolds $X$, we translate the notion of `equivalence' of periodic tangles to quotient maps of $X_\infty$. We define this equivalence as follows:

\begin{dfn}
\label{def:motif_equivalence}
A \emph{motif} of a periodic tangle $L_{\infty}$ is a link in $X$ whose lift in the universal cover $X_\infty$ is equivalent to $L_{\infty}$. 
\end{dfn}

If two periodic tangles $L_{\infty}$ and $L'_{\infty}$ are equivalent and $L$ is a motif of $L_{\infty}$, then we say that $L$ is also a motif of $L'_{\infty}$. A consequence of Definition~\ref{def:motif_equivalence} is that any finite cover of a motif of $L_{\infty}$ is also a motif of $L_{\infty}$. Thus, finite covers of motifs are equivalent under this definition. One might now ask what other types of motifs give rise to equivalent periodic tangles.
Just as ambient isotopies between links define equivalence, we can also study homeomorphisms of $X\rightarrow X$. This can include trivial diffeomorphisms and homeomorphisms that correspond to Reidemeister moves, but for these periodic manifolds, we have an additional class of homeomorphisms. In doubly-periodic tangles, it has been shown that Dehn twists are homeomorphisms of $T^{2} \times I \to T^{2} \times I$ that give rise to equivalent periodic tangles \cite{GMO07,DLM26}. This can be extended to Dehn twists in $S^{1} \times D^{2}$ \cite{SZ25} and $T^3$ \cite{AEM25}.

Next, we introduce the notion of \emph{admissible} homeomorphism.
More precisely, a homeomorphism $h \colon X \to X$ for $X = S^{1} \times D^{2}$, $T^{2} \times I$, or $T^3$ is called \emph{admissible} if $h$ satisfies the following: 
\begin{enumerate}
\item The homeomorphism $h$ preserves the orientation of $X$. 
\item If $X = S^{1} \times D^{2}$, the homeomorphism $h$ is isotopic to the identity. 
\item If $X = T^{2} \times I$, the homeomorphism $h$ preserves the boundary components (i.e. the front and back). 
\end{enumerate}

We now have the tools to  translate the notion of equivalence between periodic tangles to the language of homeomorphisms of their motifs. Two periodic tangles are equivalent if and only if two motifs of theirs are related by admissible homeomorphisms. 
If $X = T^{2} \times I$, the admissible homeomorphisms include the classical equivalence relation of links up to isotopy in $X$ and up to Dehn twists of $X$, as well as the equivalence between two finite covers of the same motif. This leads us to the following proposition:

\begin{prop} 
Let $L_{0, \infty}$ and $L_{1, \infty}$ be two periodic tangles and $L_{0}$ and $L_{1}$ be two of their motifs, respectively, in $X$. Then $L_{0, \infty}$ and $L_{1, \infty}$ are equivalent if and only if there exist finite covering maps $P_{0}, P_{1} \colon X \to X$ such that there is an admissible homeomorphism $h \colon (X, P_{0}^{-1}(L_{0})) \to (X, P_{1}^{-1}(L_{1}))$. 
\end{prop}
\begin{proof}
Suppose that there exist finite covering maps $P_{0}, P_{1} \colon X \to X$ such that there is an admissible homeomorphism $h \colon (X, P_{0}^{-1}(L_{0})) \to (X, P_{1}^{-1}(L_{1}))$. 
Then $h$ lifts to a homeomorphism $h_{\infty} \colon (X_{\infty}, L_{0, \infty}) \to (X_{\infty}, L_{1, \infty})$ that is equivariant under the translational periodicity and isotopic to the identity. 
Hence $L_{0, \infty}$ and $L_{1, \infty}$ are equivalent. 

Conversely, suppose that $L_{0, \infty}$ and $L_{1, \infty}$ are equivalent. 
Then, by taking composition with an affine transformation, we may assume that there is a homeomorphism $h_{\infty} \colon (X_{\infty}, L_{0, \infty}) \to (X_{\infty}, L_{1, \infty})$ that is equivariant under the translational periodicity of $\bbZ^{d}$ ($d = 1$, $2$, or $3$) acting on $X_{\infty}$ and isotopic to the identity. 
Here $(X, L_{i})$ for each $i = 0,1$ is the quotient of $(X_{\infty}, L_{i, \infty})$ by a lattice $\Lambda_{i} \subset \bbR^{d}$. 
If $\Lambda_{i} \subset \bbQ^{d}$ for both $i = 0,1$, there is a lattice $\widetilde{\Lambda} \subset \bbZ^{d}$ such that $\widetilde{\Lambda} \subset \Lambda_{0}$ and $\widetilde{\Lambda} \subset \Lambda_{1}$. 
 Let $P_{i} \colon X \to X$ denote the finite covering maps that correspond to $\widetilde{\Lambda} \subset \Lambda_{i}$. 
Then $h_{\infty}$ induces an admissible homeomorphism $h \colon (X, P_{0}^{-1}(L_{0})) \to (X, P_{1}^{-1}(L_{1}))$. 

Otherwise the closure of the subgroup of $\bbR^{d}$ generated by $\bbZ^{d}$ and $\Lambda_{i}$ for $i = 0$ or $1$ contains a 1-dimensional subspace $V$. 
Since $(X_{\infty}, L_{i, \infty})$ is invariant under the action of $V$, 
the periodic tangle $L_{i, \infty}$ consists of parallel straight lines in $X_{\infty}$. 
In this case, there are finite covering maps $P_{i} \colon X \to X$ such that $P_{0}^{-1}(L_{0})$ and $P_{1}^{-1}(L_{1})$ are isotopic to parallel closed geodesics in $X$ of the same component number. 
Hence again there is an admissible homeomorphism $h \colon (X, P_{0}^{-1}(L_{0})) \to (X, P_{1}^{-1}(L_{1}))$. 
\end{proof}

\begin{rem}
\label{rem:diagram}
In \cite{DLM26}, the doubly periodic tangle equivalence was investigated by considering motifs along with their supporting lattices. 
A generalized Reidemeister theorem for doubly periodic tangles translates ambient isotopies of doubly periodic tangles in $\bbR^{2} \times I$ to the diagrammatic picture of motifs. It proves that two doubly periodic tangles are equivalent if and only if two motif diagrams are related by a finite sequence of moves, which includes 
\begin{enumerate}
\item parallel shifts of the longitude-meridian pair of the torus, 
\item local isotopy moves, 
\item Dehn twists (the action of $SL(2,\bbZ)$ on $T^{2}$), and 
\item finite covering maps (called scale equivalence). 
\end{enumerate}
\end{rem}

\begin{ex}
\label{ex:motif}
Figure~\ref{fig:diagram} illustrates examples of the equivalence relations listed in Remark~\ref{rem:diagram}.
On the top left, a doubly periodic (DP) diagram and its quotient under a periodic lattice, that is, a link diagram in the (flat) torus. On the top right, two motif diagrams of the same DP diagram related by Reidemeister moves, as well as their corresponding finite covers, such that periodicity is preserved. At the level of the corresponding link in the thickened torus, this equivalence relation between finite covers is one of the main results of this paper. 
On the bottom, equivalent motifs of the same DP diagram satisfying the following relations: 
\begin{itemize}
\item (a) is a double cover of (b), 
\item (b) and (c) are related by Dehn twists, 
\item (c) is a double cover or (d), and 
\item (d) and (e) are related by Dehn twists. 
\end{itemize}
These equivalence relations are also considered from a 3-manifold perspective and highlight the open problem of finding a `minimal' motif, which is also addressed in the last section of this paper. 
\end{ex}

\begin{figure}[ht]
\centerline{\includegraphics[width=5in]{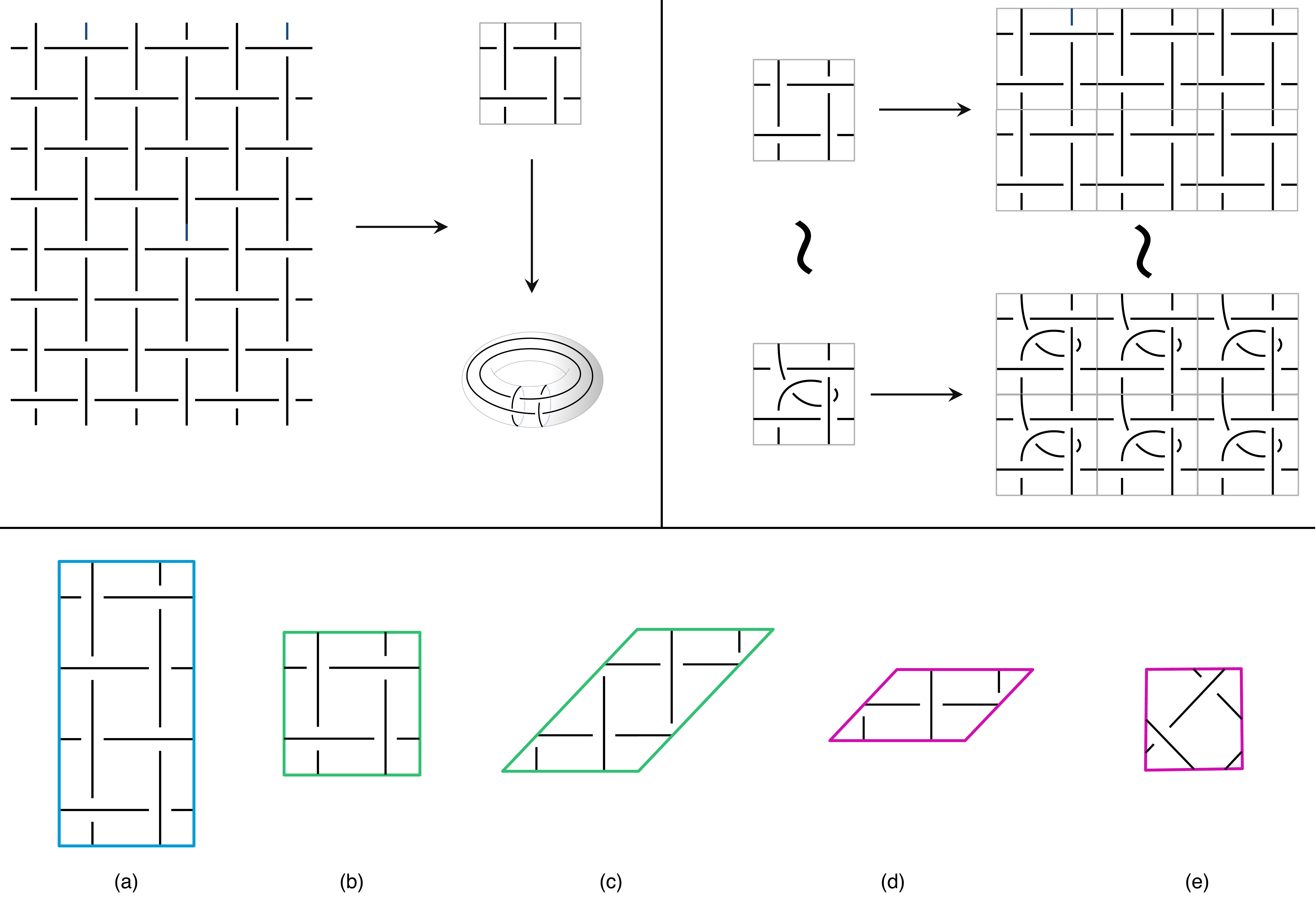}}
\vspace*{8pt}
\caption{\label{fig:diagram} 
On the top left, a doubly periodic (DP) diagram and a corresponding (flat) motif diagram. On the top right, motif diagrams and their finite covers are related by Reidemeister moves. On the bottom, motifs of the same DP diagrams related by finite covers and Dehn twists.}
\end{figure}

Given an equivalence relation for periodic tangles, a natural extension of classical knot theory is to consider classifications of periodic tangles. However, many existing topological invariants used to study classical link and knot equivalence do not yield invariant results when two non-isotopic motifs are compared. For example, the hyperbolic volumes of a motif and a finite cover of that motif will not be the same. Yet both motifs lead to equivalent periodic tangles. To address this, we introduce the definition of a \emph{minimal motif}. We will use this definition to prove the main result of Section~\ref{section:minimal}.

\begin{dfn}
\label{def:minimal-motif}
A motif $L$ of a periodic tangle $L_{\infty}$ is said to be \emph{minimal} if it is not a non-trivial finite cover of another motif. 
In other words, $L$ is minimal if it satisfies the following property. Suppose that $L'$ is a link in $X$ that is equivalent to $L$, then every finite covering map $(X,L) \to (X,L')$ is a homeomorphism.
\end{dfn}

In \cite{DLM26}, another definition of a minimal motif is given in terms of the associated periodic lattice.
In our definition, we have to be aware that a periodic tangle may become more symmetric by a periodic isotopy. 
The existence and uniqueness of a minimal motif are non-trivial questions, 
which we will discuss in Section~\ref{section:minimal}. 
Note that seeking systematic ways to find a minimal motif is still an open problem. 

\begin{rem}
In Example \ref{ex:motif}, the motif (e) is minimal. 
The hyperbolic structures of the complements of the motifs (a)-(e) are obtained by gluing regular ideal octahedra (see \cite{CKP16} for details). 
In particular, the complement of (e) can be decomposed into two regular ideal octahedra, and its volume is $2v_{\mathrm{oct}}$, where $v_{\mathrm{oct}} = 3.66...$ is the volume of a regular ideal octahedron. 
Since the complement of every motif of a hyperbolic doubly periodic tangle has at least three cusps, 
the volume is more than $v_{\mathrm{oct}}$ by \cite{Agol10}. 
Hence the complement of (e) cannot be a non-trivial finite cover of any other motif. 
\end{rem}

\subsection{Complements of links in 3-manifolds}

We will investigate links using 3-manifold theory. 
In this subsection, we introduce terminology we will need going forward. 
For a link $L$ in a compact 3-manifold $X$, we define the \emph{complement} $X \cut L$ of $L$ as the compact 3-manifold obtained by removing an open tubular neighborhood $N(L)$ of $L$ from $X$.

To describe group actions on the complement $X \cut L$, 
we will consider its decomposition. 
It is well known that an oriented compact 3-manifold admits a unique decomposition by the prime decomposition and the JSJ decomposition. 
The existence of the prime decomposition was proven by Kneser \cite{Kneser29}, 
and its uniqueness was proven by Milnor \cite{Milnor62}. 
The JSJ decomposition is due to Jaco, Shalen, and Johannson \cite{JS79, Johannson06}. 
We will explain these two notions in the following subsections, 
and we refer to \cite{Hatcher07} for more details.

A \emph{meridian} of $L$ is a simple closed curve on the boundary of $N(L)$ that bounds a disk in $N(L)$. 
For links $L_{0}$ and $L_{1}$ in $X$, 
if a homeomorphism between $X \cut L_{0}$ and $X \cut L_{1}$ maps any meridian to a meridian, 
it can be extended to a homeomorphism between pairs $(X, L_{0})$ and $(X, L_{1})$.

For a properly embedded surface $S$ in a compact 3-manifold $M$, we define $M \cut S$ as the compact 3-manifold obtained by removing an open tubular neighborhood of $S$ from $M$. 
We say that $M \cut S$ is the manifold obtained by decomposing $M$ along $S$. 
Then a component of $M \cut S$ is called a \emph{piece}.

A \emph{longitude} of $S^{1} \times D^{2}$ is a simple closed curve isotopic to a curve $S^{1} \times \{ \ast \}$ for any point $\ast$ in the boundary $S^{1} \times \partial D^{2}$. 
A \emph{longitude} of a component $K$ of a link in $S^{1} \times D^{2}$ or $T^{2} \times I$ is a simple closed curve in the boundary of an open tubular neighborhood $N(K)$ whose linking number with $K$ is zero. 
Here the linking number is defined in the usual way using projections $S^{1} \times D^{2} \to S^{1} \times I$ and $T^{2} \times I \to T^{2}$.

\subsection{Prime decomposition}

The connected sum $M_{0} \# M_{1}$ of 3-manifolds $M_{0}$ and $M_{1}$ is obtained by removing a ball in each $M_{i}$ and gluing them along the resulting boundary spheres. 
If $M_{0}$ and $M_{1}$ are connected and oriented, 
the oriented connected sum $M_{0} \# M_{1}$ is uniquely determined. 
The connected sum operation is commutative and associative, 
and $S^{3}$ is the identity element for the connected sum. 
A 3-manifold $M$ is \emph{prime} 
if $M = P \# Q$ implies that $P = S^{3}$ or $Q = S^{3}$. 
A 3-manifold $M$ is \emph{irreducible} 
if every embedded sphere in $M$ bounds a 3-ball. 
In the context of our paper, the 3-manifolds $S^{3}$, $S^{1} \times D^{2}$, $T^{2} \times I$, and $T^{3}$ are irreducible. 
If an orientable 3-manifold $M$ is prime and \emph{not} irreducible, 
then $M = S^{1} \times S^{2}$. 
The \emph{prime decomposition} of a compact orientable 3-manifold $M$ is a decomposition of the form $M = P_{1} \# \dots \# P_{n}$, 
where each $P_{i}$ is a prime 3-manifold.
Moreover, a prime decomposition is unique up to permutation and connected sums with additional copies of $S^3$.

Let $L$ be a link in an irreducible 3-manifold $X$. 
We say that $L$ splits to links $L^{\prime}$, $L_{1}$, \dots, $L_{m}$ 
if $L = L^{\prime} \sqcup L_{1} \sqcup \dots \sqcup L_{m}$ and there are disjoint balls $B_{1}, \dots, B_{m}$ in $X$ disjoint from $L^{\prime}$ such that each $B_{i}$ contains $L_{i}$. 
If $m \geq 1$, then $L$ is called a \emph{split} link. 
The link $L$ is split 
if and only if the complement $X \cut L$ is not irreducible. 
The following lemma is a consequence of the prime decomposition of the complement $X \cut L$ of a link $L$. 

\begin{lem}
\label{lem:prime}
Let $L$ be a link in an irreducible compact orientable 3-manifold $X$. 
Then $L$ maximally splits to links $L^{\prime}$, $L_{1}$, \dots, $L_{m}$, 
where $L^{\prime}$ is a (possibly empty) non-split link, and $L_{1}$, \dots, $L_{m}$ are contained in disjoint 3-balls. 
Moreover, this splitting is unique up to permutations of $L_{1}$, \dots, $L_{m}$. 
\end{lem}

\begin{proof}
Consider the prime decomposition $X \cut L = P_{0} \# P_{1} \# \dots \# P_{m}$ such that $P_{i} \neq S^{3}$. 
There exist disjoint spheres $S_{1}$, \dots, $S_{m}$ in $X \cut L$ along which the connected sums in the prime decomposition occur. 
Since $X$ is irreducible, each $S_{i}$ bounds a ball $B_{i}$ in $X$. 
If $B_{i}$ is contained in $B_{j}$, 
we can replace $S_{j}$ with a smaller one so that $B_{j}$ is disjoint from $B_{i}$. 
Hence we may assume that the balls $B_{1},\ldots,B_m$ are disjoint. 
Then by permuting the indices, we have $P_{0} = (X \cut L) \cup B_{1} \cup \dots \cup B_{m}$, and $P_{i}$ is obtained by capping $X \cut L \cap B_{i}$ with a ball for $1 \leq i \leq m$. 
Let $L_{i} = L \cap B_{i}$ and $L^{\prime} = L \setminus (L_{1} \sqcup \dots \sqcup L_{m})$. 
Then $L$ splits to links $L^{\prime}$, $L_{1}$, \dots, $L_{m}$. 
The maximality and uniqueness follow from those of the prime decomposition. 
In particular, $L^{\prime}$ is non-split. 
\end{proof}

\subsection{Seifert fibered 3-manifolds}

For coprime integers $p$ and $q \neq 0$, a \emph{$(p,q)$-model Seifert fibration} on $S^{1} \times D^{2}$ is a decomposition of $S^{1} \times D^{2}$ into disjoint circles, called \emph{fibers}, obtained from the segments $[0,1] \times \{ x \}$ in $[0,1] \times D^{2}$ by gluing $\{ 0 \} \times D^{2}$ and $\{ 1 \} \times D^{2}$ via a rotation of $2\pi p/q$. 
The $(p,1)$-model Seifert fibration is the same as the fibration given by the usual product structure on $S^1\times D^2$ (sometimes referred to as a \emph{product fibration}) -- that is, every fiber is a circle over a point in the disk and the resulting space is $S^{1} \times D^{2}$.

A \emph{Seifert fibration} on a 3-manifold $M$ is a decomposition of $M$ into disjoint circles, also called \emph{fibers}, such that each fiber has a tubular neighborhood with a model Seifert fibration. 
A 3-manifold $M$ is called a \emph{Seifert fibered 3-manifold} if $M$ admits a Seifert fibration. 
A fiber in a Seifert fibered 3-manifold is called a \emph{regular} fiber if it has a tubular neighborhood with a product fibration. 
Otherwise it is called a \emph{singular} fiber.

The classification of Seifert fibered 3-manifolds was established many years ago, and here, we use the notation for orientable Seifert fibered 3-manifolds following \cite{Budney06, Hatcher07}. 
For $g \in \bbZ$ and $b \in \bbZ_{\geq 0}$, let $\Sigma_{g,b}$ denote the orientable compact surface of genus $g$ with $b$ boundary components if $g \geq 0$. 
Let $\Sigma_{g,b}$ denote the non-orientable compact surface that is the connected sum of $-g$ copies of $\bbR P^{2}$ and $\Sigma_{0,b}$ if $g < 0$. 
Let $M(g, b; \alpha_{1} / \beta_{1}, \dots, \alpha_{k} / \beta_{k})$ denote the orientable Seifert fibered 3-manifold that is fibered over $\Sigma_{g,b}$ with at most $k$ singular fibers and fiber data $\alpha_{i} / \beta_{i} \in \bbQ$. 
More precisely, the 3-manifold $M(g, b; \alpha_{1} / \beta_{1}, \dots, \alpha_{k} / \beta_{k})$ is obtained from the orientable $S^{1}$-bundle over $\Sigma_{g,b+k}$ by Dehn filling along $k$ boundary components 
so that the meridians of attached solid tori are glued to the slopes $\alpha_{i} / \beta_{i}$. 
Here the slope $\alpha_{i} / \beta_{i}$ on the torus $T_{i}$ for coprime integers $\alpha_{i}$ and $\beta_{i}$ is represented by $\alpha_{i} \lambda_{i} + \beta_{i} \mu_{i} \in H_{1}(T_{i}, \bbZ)$ for an oriented pair of a fiber $\lambda_{i}$ and a boundary component of a section $\mu_{i}$. 
The Seifert fibration on $M(g, b; \alpha_{1} / \beta_{1}, \dots, \alpha_{k} / \beta_{k})$ is obtained by extending the product fibration over $\Sigma_{g,b+k}$. 
The 3-manifolds with non-unique Seifert fibrations are classified. 
We state the classification for bounded orientable Seifert fibered 3-manifolds. 

\begin{prop}[\cite{Jaco80} Theorem VI. 18]
\label{prop:unique-seifert}
Let $M$ be an orientable Seifert fibered 3-manifold with non-empty boundary. 
Then a Seifert fibration on $M$ is unique up to isotopy 
unless $M$ is the solid torus $S^{1} \times D^{2}$, the thickened torus $T^{2} \times I$, or the twisted $I$-bundle over the Klein bottle $S^{1} \tilde{\times} S^{1} \tilde{\times} I$. 

The solid torus $S^{1} \times D^{2}$ admits the model Seifert fibrations of $M(0,1; \alpha / \beta)$. 
The thickened torus $T^{2} \times I$ admits the product fibration of $M(0,2;)$, which is unique up to isomorphism. 
The twisted $I$-bundle over the Klein bottle $S^{1} \tilde{\times} S^{1} \tilde{\times} I$ is homeomorphic to $M(0,1;1/2,1/2)$ and $M(-1,1;)$. 
\end{prop}

\subsection{Hyperbolic 3-manifolds}

A hyperbolic manifold is a complete Riemannian manifold of constant sectional curvature $-1$. 
We refer to \cite{BP92, Ratcliffe19, Thurston78} for details on hyperbolic 3-manifolds.

An compact irreducible 3-manifold $M$ is \emph{(algebraically) atoroidal} 
if any $\bbZ \times \bbZ$-subgroup of $\pi_{1}(M)$ is contained in a peripheral subgroup 
(i.e. an image of the fundamental group of a boundary component). 
A finite-volume hyperbolic 3-manifold is the interior of an atoroidal 3-manifold. 
Conversely, Thurston's hyperbolization \cite{Thurston82} implies that the interior of an atoroidal 3-manifold with non-empty boundary consisting of tori is homeomorphic to a finite-volume hyperbolic 3-manifold 
unless it is $T^{2} \times I$. 
From now on, a compact 3-manifold whose interior is a hyperbolic 3-manifold is also called a hyperbolic 3-manifold. 

Each boundary component of a finite-volume hyperbolic 3-manifold is called a \emph{cusp}. 
A neighborhood of a cusp is isometric to a quotient of a horoball in the hyperbolic 3-space by a $\bbZ \times \bbZ$-subgroup of the fundamental group. 
Such a neighborhood is called a \emph{horocusp}. 
The boundary of a horocusp is isometric to a Euclidean torus. 
Then each cusp inherits a unique Euclidean structure up to scaling. 

An important property of a hyperbolic 3-manifold is the Mostow rigidity theorem (also known as the Mostow--Prasad rigidity theorem) \cite{Mostow73, Prasad73}, which states that the hyperbolic structure on an atoroidal 3-manifold is unique. 
More precisely, suppose that $M_{0}$ and $M_{1}$ are finite-volume hyperbolic 3-manifolds. 
If there exists a homeomorphism $f \colon M_{0} \to M_{1}$, 
then $f$ is isotopic to a unique isometry from $M_{0}$ to $M_{1}$. 
In fact, a homotopy equivalence between finite-volume hyperbolic manifolds of dimension at least three is homotopic to an isometry. 
The fact that homotopic homeomorphisms between Haken 3-manifolds are isotopic was shown by Waldhausen \cite{Waldhausen68} (and shown by Gabai, Meyerhoff, and Thurston \cite{GMT03} for general 3-manifolds). 

A link $L$ in a 3-manifold $X$ with (possibly empty) boundary consisting of tori is called \emph{hyperbolic} 
if the interior of the complement $X \cut L$ admits a finite-volume hyperbolic structure.
Due to Thurston's hyperbolization, plenty of links are hyperbolic. 
The hyperbolic structure of each hyperbolic link is unique by the Mostow rigidity theorem. 
Hence the hyperbolic structure induces many geometric invariants, such as the hyperbolic volume. 
In this paper, we use the Mostow rigidity theorem to consider symmetries of hyperbolic links.

\subsection{JSJ decomposition}

The JSJ decomposition is a standard decomposition of a compact irreducible 3-manifold along incompressible tori as follows. 
A properly embedded surface $S$ in a 3-manifold $M$ is \emph{2-sided} 
if its normal bundle is trivial. 
If $S$ and $M$ are orientable, then $S \subset M$ is always 2-sided. 
A 2-sided surface $S$ is \emph{compressible} 
if there is a disk $D \subset M$ 
such that $D \cap S = \partial D$ and $\partial D$ does not bound a disk in $S$. 
Such a disk $D$ is called a \emph{compressing disk}. 
In this case, we can compress the surface $S$ along $D$ to a new surface 
obtained by removing a tubular neighborhood of $\partial D$ from $S$ 
and capping it with two disks parallel to $D$. 
A 2-sided surface $S$ other than $S^{2}$ or $D^{2}$ is \emph{incompressible} 
if it is not compressible. 
The loop theorem \cite[Theorem 3.1]{Hatcher07} implies that $S \subset M$ is incompressible 
if and only if the induced homomorphism $\pi_{1}(S) \to \pi_{1}(M)$ is injective.

Let $M$ be a compact irreducible 3-manifold with (possibly empty) boundary consisting of tori. 
The \emph{JSJ decomposition} (also known as the torus decomposition) of $M$ 
is decomposition along (possibly empty) disjoint incompressible tori 
that satisfies the following: 
\begin{itemize}
\item each piece after decomposition is Seifert fibered or atoroidal, and 
\item the collection of tori is minimal under the above condition. 
\end{itemize}
Such tori are unique up to isotopy, and called the \emph{JSJ tori}. 
Each piece after the JSJ decomposition is called a \emph{JSJ piece}. 
Thurston's hyperbolization \cite{Thurston82} implies that 
an atoroidal 3-manifold with boundary is hyperbolic. 
The geometrization theorem proven by Perelman~\cite{Perelman02, Perelman03a} implies that 
a closed atoroidal 3-manifold is spherical or hyperbolic. 
Moreover, a spherical 3-manifold admits a Seifert fibration (see \cite{Scott83} for details). 
Hence each JSJ piece is Seifert fibered or hyperbolic. 
In this paper, however, the closed case is not necessary. 

We will investigate the JSJ decomposition of a link complement $X \cut L$ 
in Section~\ref{section:jsj}. 
Now we show the behavior of the JSJ decomposition for a finite cover. 
Note that a link complement $X \cut L$ which we will consider 
does not contain $S^{1} \tilde{\times} S^{1} \tilde{\times} I$. 

\begin{lem}
\label{lem:jsj-cover}
Let $M$ be an irreducible orientable 3-manifold with (possibly empty) boundary consisting of tori. 
Let $\widetilde{M}$ be a finite cover of $M$. 
Then the JSJ tori of $\widetilde{M}$ is the preimage of the JSJ tori of $M$ 
unless either a JSJ piece of $M$ is $S^{1} \tilde{\times} S^{1} \tilde{\times} I$ or $M$ is a Sol 3-manifold. 
\end{lem}

If there are at least two JSJ pieces of $M$, 
a JSJ piece of $M$ is $S^{1} \tilde{\times} S^{1} \tilde{\times} I$, 
and its preimage in $\widetilde{M}$ is $T^{2} \times I$, 
then this $T^{2} \times I$ is not a JSJ piece. 
In this case, the minimality of JSJ tori is violated. 

A Sol 3-manifold is a sufficiently twisted $T^{2}$-bundle over $S^{1}$. 
A fiber of a Sol 3-manifold is a unique JSJ torus. 
A finite cover $\widetilde{M}$ of a Sol 3-manifold $M$ is also a Sol 3-manifold. 
Then the preimage of the JSJ torus of $M$ is not connected in general, 
and the minimality of JSJ tori is violated. 

\begin{proof}
Let $\mathcal{T}$ denote the collection of JSJ tori of $M$, 
and let $\widetilde{\mathcal{T}}$ denote the preimage of $\mathcal{T}$ in $\widetilde{M}$. 
If a 3-manifold is Seifert fibered or hyperbolic, 
a finite cover of it is also in the same class. 
Hence the tori in $\widetilde{\mathcal{T}}$ decomposes $\widetilde{M}$ 
into Seifert fibered or hyperbolic pieces. 

Suppose that $\widetilde{\mathcal{T}}$ is not the collection of JSJ tori. 
Then $\widetilde{\mathcal{T}}$ has a torus $\widetilde{T}$ 
between two Seifert fibered pieces $\widetilde{M_{0}}$ and $\widetilde{M_{1}}$ 
of $\widetilde{M} \cut \widetilde{\mathcal{T}}$ 
whose fibers coincide in $\widetilde{T}$. 
Let $T$, $M_{0}$, and $M_{1}$ respectively denote 
the image of $\widetilde{T}$, and $\widetilde{M_{0}}$ and $\widetilde{M_{1}}$. 
Then $T$ is a JSJ torus, and $M_{0}$ and $M_{1}$ are JSJ pieces of $M$. 
Note that $\widetilde{M_{0}} = \widetilde{M_{1}}$ and $M_{0} = M_{1}$ are possible. 
For $i=0,1$, a Seifert fibration on $M_{i}$ lifts to that of $\widetilde{M_{i}}$. 
If the Seifert fibrations on $\widetilde{M_{i}}$ and $M_{i}$ are unique, 
then the fibers of $M_{0}$ and $M_{1}$ coincide in $T$, 
which contradicts the fact that $T$ is a JSJ-torus. 
Hence at least one of $\widetilde{M_{i}}$ and $M_{i}$ 
is $S^{1} \times D^{2}$, $T^{2} \times I$, or $S^{1} \tilde{\times} S^{1} \tilde{\times} I$ 
by Proposition~\ref{prop:unique-seifert}. 
Since these three 3-manifolds finitely cover no other orientable 3-manifolds, 
$M_{0}$ or $M_{1}$ 
is $S^{1} \times D^{2}$, $T^{2} \times I$, or $S^{1} \tilde{\times} S^{1} \tilde{\times} I$. 
Moreover, $S^{1} \times D^{2}$ and $T^{2} \times I$ are not JSJ pieces 
unless $M$ itself is $S^{1} \times D^{2}$, $T^{2} \times I$, or a Sol 3-manifold. 
Therefore $M_{0}$ or $M_{1}$ is $S^{1} \tilde{\times} S^{1} \tilde{\times} I$, or $M$ is a Sol 3-manifold. 
\end{proof}

To consider the JSJ decomposition of a link complement $X \cut L$, 
we need to describe regions bounded by an embedded torus in $X$. 
We say that a \emph{knotted hole ball} in a 3-manifold $X$ is the complement of a properly embedded arc in a 3-ball contained in $X$ that is not a solid torus.

\section{JSJ decomposition of link complements}
\label{section:jsj}

\subsection{Links in the solid torus}

In this subsection, we consider links in the solid torus $S^{1} \times D^{2}$. 
We first define elementary links in the solid torus, 
which are analogous to the torus links in the 3-sphere. 
Let $S^{1} = \bbR / \bbZ$ and $D^{2} = \{ (x, y) \in \bbR^{2} \mid x^{2} + y^{2} \leq 1 \}$. 
For $(p, q) \in \bbZ^{2} \setminus \{ (0,0) \}$, 
define 
\begin{align*}
T_{0}(p, q)  = \{ & \left( p^{\prime} t, \frac{1}{2} \cos 2\pi \left( q^{\prime} t + \frac{k}{p} \right), \frac{1}{2} \sin 2\pi \left( q^{\prime} t + \frac{k}{p} \right) \right) \in S^{1} \times D^{2} \mid \\ 
& t \in [0,1], k = 0, \dots, d-1 \}, 
\end{align*}
where $d = \gcd(p,q)$, $p^{\prime} = p/d$, and $q^{\prime} = q/d$. 
The link $T_{0}(p,q)$ consists of $d$ rotated copies of $T_{0}(p^{\prime}, q^{\prime})$. 
Define $T_{1}(0,0) = S^{1} \times \{ (0,0) \}$ and 
\[
T_{1}(p, q) = T_{0}(p, q) \cup T_{1}(0,0) 
\]
for $(p, q) \in \bbZ^{2} \setminus \{ (0,0) \}$. 
Figure~\ref{fig1} shows the links $T_{0}(3,4)$ and $T_{1}(3,4)$. 

\begin{figure}[ht]
\centerline{\includegraphics[width=5in]{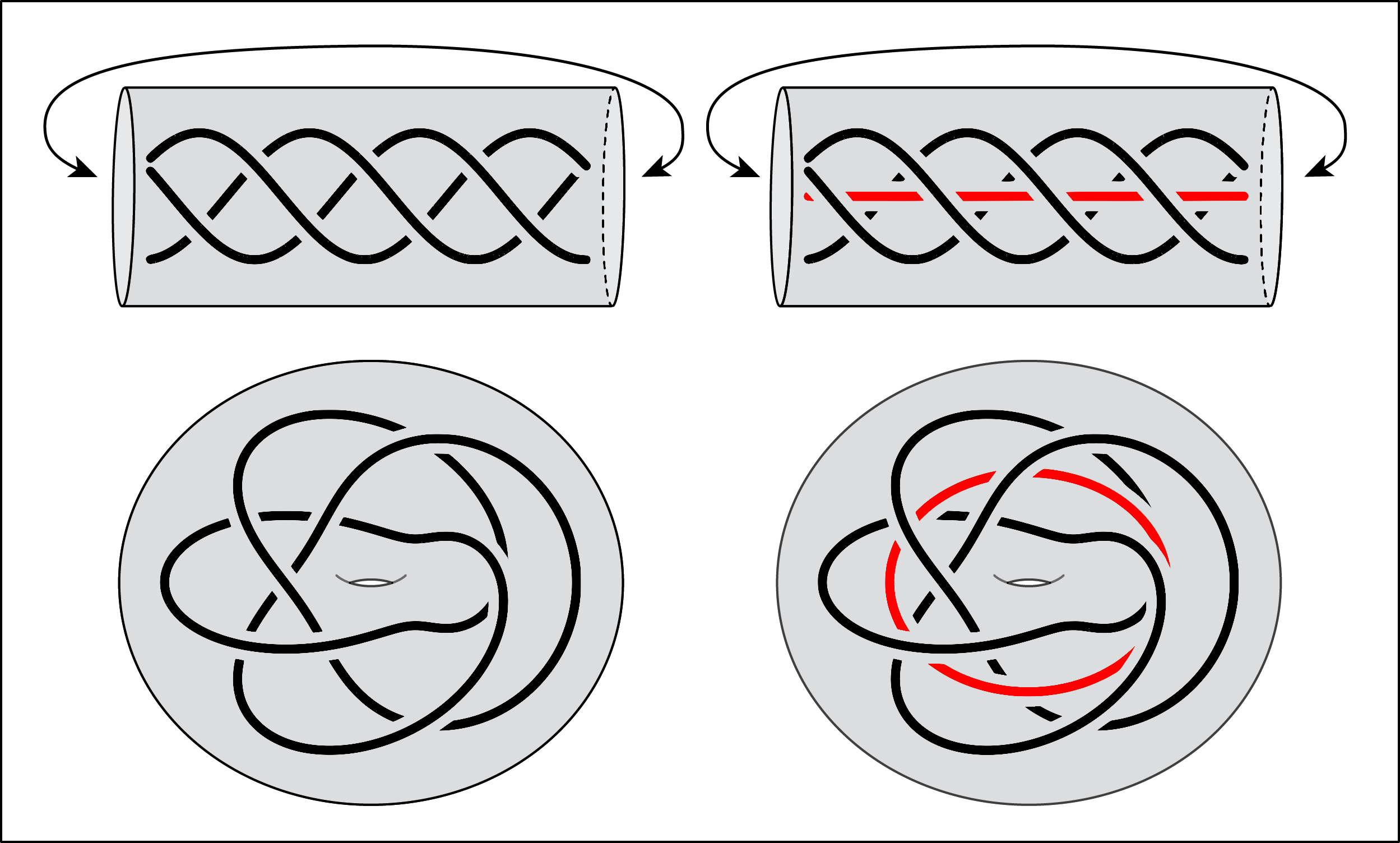}}
\vspace*{8pt}
\caption{\label{fig1} 
The links $T_{0}(3,4)$ and $T_{1}(3,4)$.}
\end{figure}

Since the links are not oriented, 
the links $T_{l}(p,q)$ and $T_{l}(-p,-q)$ for $l=0,1$ are identical. 
If $p$ divides $q$, then $T_{0}(p, q)$ is isotopic to $T_{1}(p-1,(p-1)q/p)$. 
In particular, $T_{0}(1,q)$ for any $q$ is isotopic to $T_{1}(0,0)$. 
We show that the other pairs in $T_{l}(p,q)$ are not isotopic. 

\begin{lem}
\label{lem:1-t-link}
For $i=0,1$, let $l_{i} \in \{ 0,1 \}$ and $p_{i} \geq 0$. 
Suppose that $p_{i}$ does not divide $q_{i}$ if $l_{i} = 0$. 
If $T_{l_{0}}(p_{0}, q_{0})$ and $T_{l_{1}}(p_{1}, q_{1})$ are isotopic, 
then $l_{0} = l_{1}$, $p_{0} = p_{1}$, and $q_{0} = q_{1}$. 
\end{lem}
\begin{proof}
Let $d_{i} = \gcd(p_{i}, q_{i})$, $p_{i}^{\prime} = p_{i} / d_{i}$, and $q_{i}^{\prime} = q_{i} / d_{i}$. 
Suppose that $l_{0} = 0$. 
Since $p_{0}$ does not divide $q_{0}$, 
any component of $T_{0}(p_{0}, q_{0})$ is not isotopic to $T_{1}(0,0)$. 
Hence $l_{1} = 0$. 
Since the number of component of the link $T_{0}(p_{i}, q_{i})$ is $d_{i}$, 
we have $d_{0} = d_{1}$. 
Since any component of $T_{0}(p_{i}, q_{i})$ represents 
$p_{i}^{\prime}$ times of a generator of $\pi_{1}(S^{1} \times D^{2})$, 
we have $p_{0}^{\prime} = p_{1}^{\prime}$. 
Moreover, the knots $T_{0}(p_{0}^{\prime}, q_{0}^{\prime})$ and $T_{0}(p_{1}^{\prime}, q_{1}^{\prime})$ are isotopic. 
Since the standard embedding of the solid torus in $\bbR^{3}$ maps $T_{0}(p,q)$ to the $(p,q)$-torus link, 
the $(p_{0}^{\prime}, q_{0}^{\prime})$-torus knot and the $(p_{1}^{\prime}, q_{1}^{\prime})$-torus knot are isotopic. 
The classification of torus knots in \cite[Theorem 7.4.3]{Murasugi96} 
implies that $q_{0}^{\prime} = q_{1}^{\prime}$. 
Hence $p_{0} = p_{1}$ and $q_{0} = q_{1}$. 

Suppose that $l_{0} = l_{1} = 1$. 
If $p_{0} = q_{0} = 0$, then clearly $p_{1} = q_{1} = 0$. 
We consider the other cases. 
If $p_{0}^{\prime} \neq 1$, 
then any component of $T_{1}(p_{0}, q_{0})$ other than the component $T_{1}(0,0)$ 
is not isotopic to $T_{1}(0,0)$. 
In this case, $T_{0}(p_{0}, q_{0})$ and $T_{1}(p_{1}, q_{1})$ are isotopic. 
Hence $p_{0} = p_{1}$ and $q_{0} = q_{1}$ by the above argument. 
If $p_{0}^{\prime} = 1$, 
then any two components of $T_{1}(p_{0}, q_{0})$ are isotopic. 
We have $d_{0} = d_{1}$, $p_{1}^{\prime} = 1$.
Moreover, the links $T_{1}(1, q_{0}^{\prime})$ and $T_{1}(1, q_{1}^{\prime})$ are isotopic. 
The standard embedding of the solid torus in $\bbR^{3}$ maps $T_{1}(1,q)$ to the $(2,q)$-torus link. 
Since the $(2,q)$-torus links are distinguished by the linking number, 
we have $q_{0}^{\prime} = q_{1}^{\prime}$. 
Hence $p_{0} = p_{1}$ and $q_{0} = q_{1}$. 
\end{proof}

The complements of $T_{l}(p,q)$ other than $T_{0}(0,q)$ ($q \neq 0$) are Seifert fibered. 
In fact, the following homeomorphisms hold: 
\begin{itemize}
\item $S^{1} \times D^{2} \cut T_{0}(p, q) \simeq M(0, \gcd(p, q)+1; \alpha / \beta)$ 
for some $\alpha, \beta\in \bbZ$, and 
\item $S^{1} \times D^{2} \cut T_{1}(p, q) \simeq M(0, \gcd(p, q)+2;)$, 
\end{itemize}
where $\gcd(0,0) = 0$. 
The complement of a link in the solid torus is homeomorphic to 
the complement of a link in the 3-sphere one of whose components is the unknot. 
Moreover, Burde and Murasugi \cite{BM70} classified the links in the 3-sphere 
whose complements are Seifert fibered (see also \cite[Proposition 3.3]{Budney06}). 
This implies the following classification of Seifert fibered links. 

\begin{prop}
\label{prop:1-seifert}
Let $L$ be a link in $S^{1} \times D^{2}$. 
Suppose that the complement $S^{1} \times D^{2} \cut L$ is Seifert fibered. 
Then $L$ is isotopic to $T_{0}(p,q)$ or $T_{1}(p,q)$. 
\end{prop}

Suppose that $L$ is a non-split link in $S^{1} \times D^{2}$. 
Let us consider the JSJ decomposition of $S^{1} \times D^{2} \cut L$. 
We classify tori in the solid torus. 

\begin{lem}
\label{lem:1-torus}
Let $T$ be an embedded torus in $S^{1} \times D^{2}$. 
Then 
\begin{itemize}
\item $T$ bounds a solid torus, or 
\item $T$ bounds the connected sum of two solid tori. 
\end{itemize}
In the latter case, $T$ bounds a solid torus or a knotted hole ball. 
\end{lem}
\begin{proof}
Since $\pi_{1}(S^{1} \times D^{2}) \cong \bbZ$ does not have a subgroup isomorphic to $\pi_{1}(T) \cong \bbZ \times \bbZ$, 
any torus $T$ in $S^{1} \times D^{2}$ is compressible. 
A sphere $S$ is obtained by compressing $T$ along a compressing disk $D$. 
The sphere $S$ bounds a ball $B$ since the solid torus is irreducible. 
If $B$ and $D$ are disjoint, then $T$ bounds a solid torus. 
If $D$ is contained in $B$, then $T$ bounds a 3-manifold obtained by attaching 
a 1-handle to $\overline{S^{1} \times D^{2} \setminus B}$, 
which is homeomorphic to the connected sum of two solid tori. 
In this case, the other region bounded by $T$ is the complement of the 1-handle in $B$, 
and so it is a solid torus or a knotted hole ball. 
\end{proof}

We say that $L$ is a \emph{satellite link} if $S^{1} \times D^{2} \cut L$ has an essential torus. 
Then the following classification of links in the solid torus 
clearly holds by Proposition~\ref{prop:1-seifert}. 

\begin{prop}
\label{prop:1-class}
Let $L$ be a non-split link in $S^{1} \times D^{2}$. 
Then $L$ is one of the following: 
\begin{itemize}
\item $T_{0}(p,q)$, 
\item $T_{1}(p,q)$, 
\item a hyperbolic link, or 
\item a satellite link. 
\end{itemize}
\end{prop}

\begin{rem}
\label{rem:splice}
An essential torus in $S^{1} \times D^{2} \cut L$ may bound a knotted hole ball. 
Even in such a case, we say that $L$ is a satellite link. 
For example, suppose that $L$ is the knot shown in Figure~\ref{fig2}. 
Then there is an essential torus $T$ in $S^{1} \times D^{2} \cut L$ unique up to isotopy. 
The two JSJ pieces are homeomorphic to 
the complement of the link $T_{1}(0,1)$ in $S^{1} \times D^{2}$ (the Seifert fibered 3-manifold $M(0,3;)$) 
and the complement of the trefoil knot in $S^{3}$ (the Seifert fibered 3-manifold $M(0,1; 1/2,1/3)$). 
\end{rem}

\begin{figure}[ht]
\centerline{\includegraphics[width=3.5in]{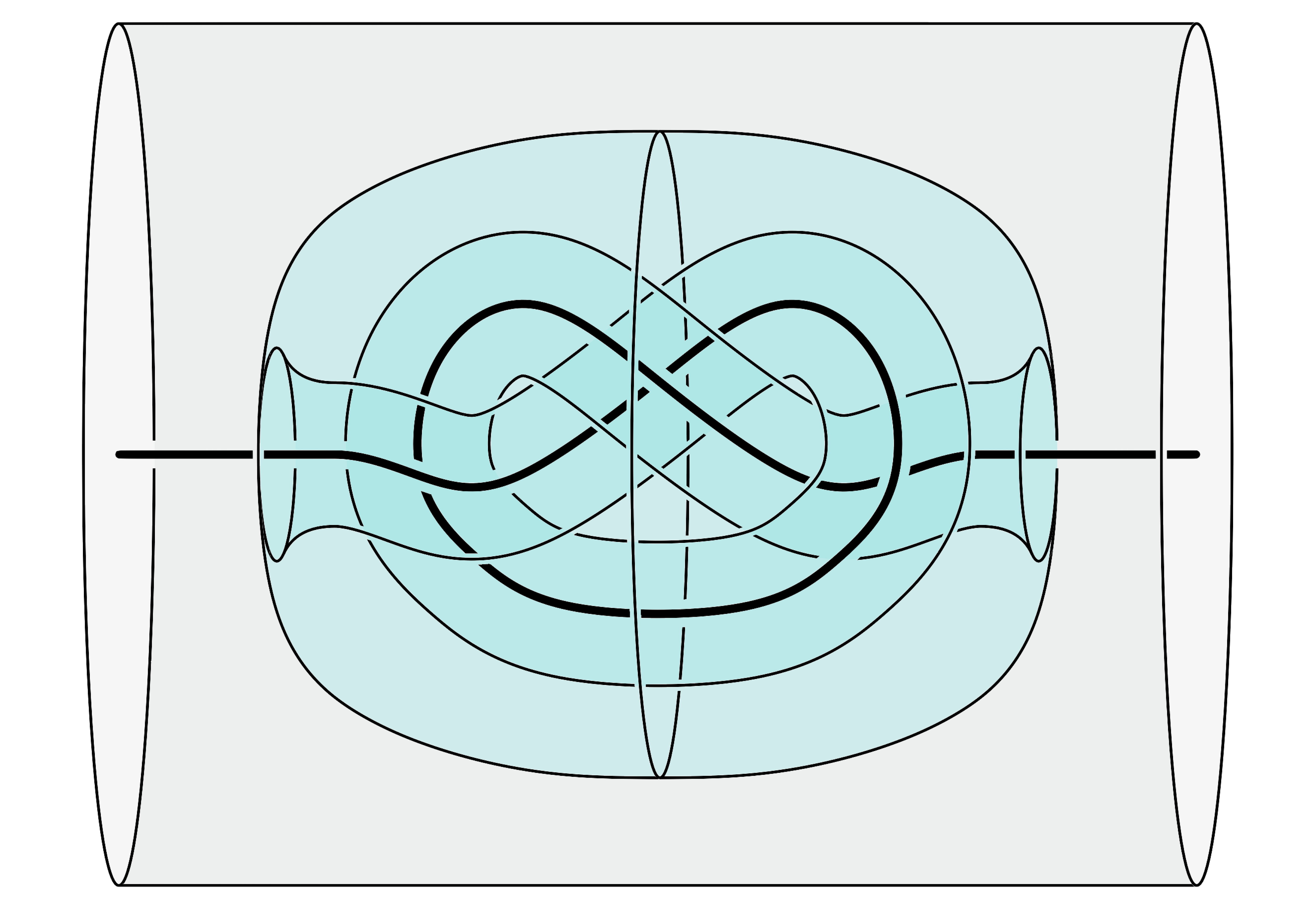}}
\vspace*{8pt}
\caption{\label{fig2} 
A link in the solid torus with an essential torus bounding a knotted hole ball.}
\end{figure}

We show that the outermost JSJ piece can be re-embedded as a link complement 
in the same manner as \cite[Proposition 2.2]{Budney06}. 

\begin{lem}
\label{lem:1-outermost}
Let $L$ be a non-split link in $S^{1} \times D^{2}$. 
Suppose that $M$ is the JSJ piece of $S^{1} \times D^{2} \cut L$ which contains $S^{1} \times \partial D^{2}$, 
called the outermost JSJ piece. 
Then $M$ is homeomorphic to the complement of a link in $S^{1} \times D^{2}$. 
Moreover, this homeomorphism can be taken to fix a neighborhood of $S^{1} \times \partial D^{2}$. 
\end{lem}
\begin{proof}
Suppose that $M$ has a boundary component $T$ which bounds a knotted hole ball 
as in Lemma~\ref{lem:1-torus}. 
Then $M$ is contained in the other region bounded by $T$, 
which is obtained by attaching a 1-handle to the complement of a 3-ball in $S^{1} \times D^{2}$. 
By unknotting this 1-handle, 
we can re-embed $M$ in $S^{1} \times D^{2}$ so that $T$ bounds a solid torus as shown in Figure~\ref{fig4}. 
Note that these attached 1-handles are disjoint by \cite[Proposition 2.1]{Budney06}. 
After these operations for all such $T$, 
the manifold $M$ is re-embedded in $S^{1} \times D^{2}$ 
so that $S^{1} \times D^{2} \setminus M$ is a tubular neighborhood of links in $S^{1} \times D^{2}$. 
\end{proof}

\begin{figure}[ht]
\centerline{\includegraphics[width=5in]{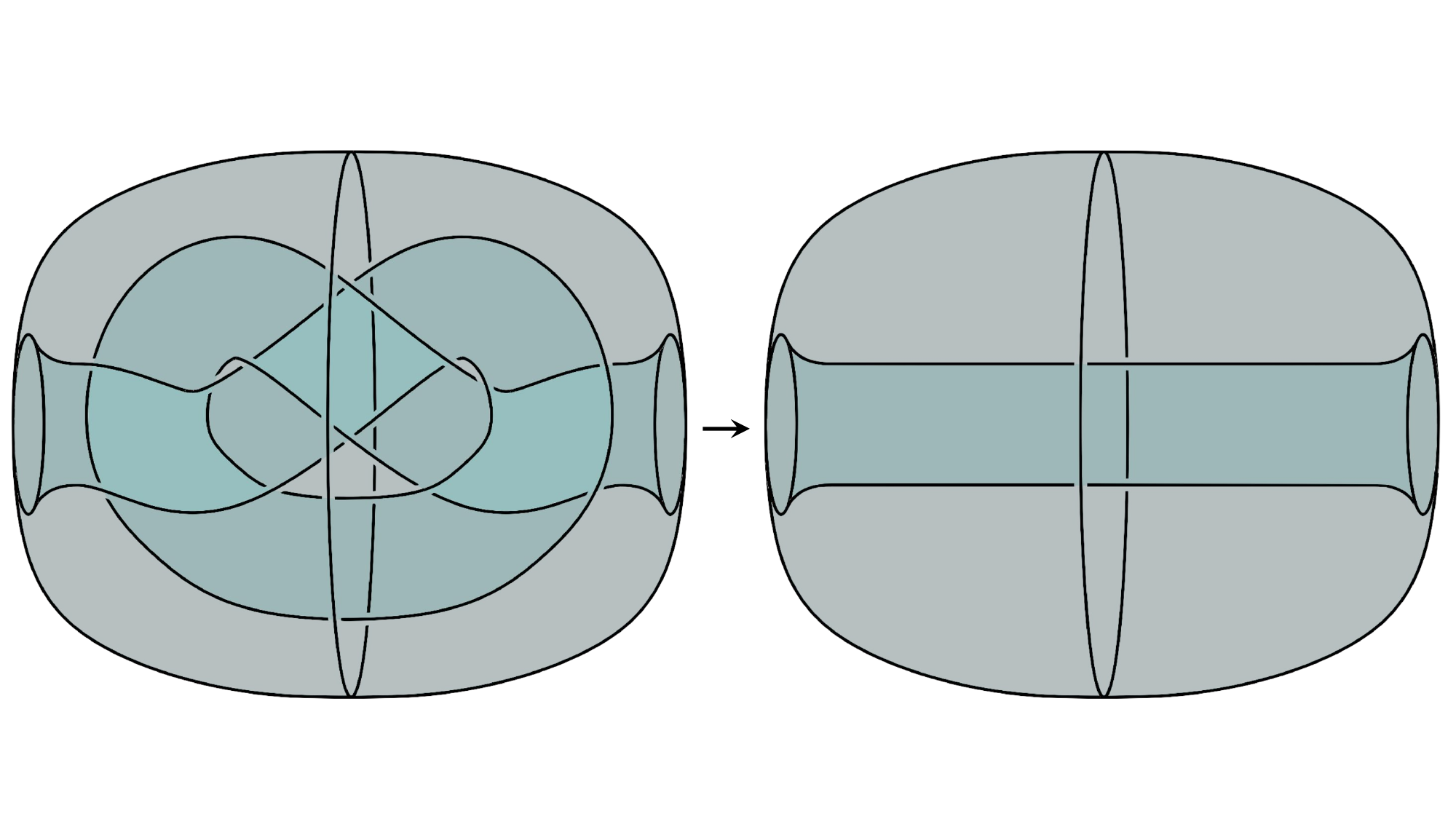}}
\vspace*{8pt}
\caption{\label{fig4} 
Re-embedding of a JSJ piece $M$.}
\end{figure}

The classification of the Seifert fibered submanifold of $S^{3}$ in \cite[Proposition 3.2]{Budney06} 
implies the following classification. 

\begin{prop}
\label{prop:1-jsj}
Let $L$ be a non-split link in $S^{1} \times D^{2}$. 
Suppose that a JSJ piece $M$ of $S^{1} \times D^{2} \cut L$ is Seifert fibered. 
Then $M$ is homeomorphic to one of the following: 
\begin{itemize}
\item $M(0, n; \alpha_{1} / \beta_{1}, \alpha_{2} / \beta_{2})$ 
	for $n \geq 1$ and $\alpha_{1} \beta_{2} - \alpha_{2} \beta_{1} = \pm 1$, 
\item $M(0, n; \alpha / \beta)$ for $n \geq 2$, or 
\item $M(0, n; )$ for $n \geq 2$. 
\end{itemize}
The Seifert fibration of $M$ is unique up to isotopy 
unless $M$ is homeomorphic to the thickened torus $M(0,2;)$. 
This exception occurs only when $L$ is isotopic to $T_{1}(0,0)$. 
\end{prop}

\subsection{Links in the thickened torus}

In this subsection, we consider links in the thickened torus $T^{2} \times I$. 
We first define elementary links in the thickened torus. 
Let $T^{2} = \bbR^{2} / \bbZ^{2}$ and $I = [0,1]$. 
For $(p,q) \in \bbZ^{2} \setminus \{ (0,0) \}$, define 
\[
T_{2}(p,q) = \{ \left( p^{\prime} t, q^{\prime}t + \frac{k}{p}, \frac{1}{2} \right) \in T^{2} \times I \mid t \in [0,1], k = 0, \dots, d-1 \}, 
\]
where $d = \gcd(p,q)$, $p^{\prime} = p/d$, and $q^{\prime} = q/d$. 
The link $T_{2}(p,q)$ consists of $d$ sliding copies of $T_{2}(p^{\prime}, q^{\prime})$. 
Since the links are not oriented, 
the links $T_{2}(p,q)$ and $T_{2}(-p,-q)$ are identical. 
We show that the other pairs in $T_{2}(p,q)$ are not isotopic.

\begin{lem}
\label{lem:2-t-link}
Suppose that $p_{0}, p_{1} \geq 0$. 
If $T_{2}(p_{0}, q_{0})$ and $T_{2}(p_{1}, q_{1})$ are isotopic, 
then $p_{0} = p_{1}$ and $q_{0} = q_{1}$. 
\end{lem}
\begin{proof}
For $i=0,1$, let $d_{i} = \gcd(p_{i}, q_{i})$, $p_{i}^{\prime} = p_{i} / d_{i}$, and $q_{i}^{\prime} = q_{i} / d_{i}$. 
Since the number of component of the link $T_{2}(p_{i}, q_{i})$ is $d_{i}$, 
we have $d_{0} = d_{1}$. 
Since any component $T_{2}(p_{i}, q_{i})$ represents 
$\pm (p_{i}^{\prime}, q_{i}^{\prime}) \in \bbZ^{2} = \pi_{1}(T^{2} \times I)$, 
we have $p_{0}^{\prime} = p_{1}^{\prime}$ and $q_{0}^{\prime} = q_{1}^{\prime}$. 
Hence $p_{0} = p_{1}$ and $q_{0} = q_{1}$. 
\end{proof}

The complement of $T_{2}(p,q)$ is homeomorphic to 
the Seifert fibered manifold $M(0, \gcd(p,q)+2;)$. 
The complement of a link in the thickened torus is homeomorphic to 
the complement of a link in the 3-sphere two of whose components form the Hopf link. 
The classification in \cite{BM70} implies the following classification of Seifert fibered links. 

\begin{prop}
\label{prop:2-seifert}
Let $L$ be a link in $T^{2} \times I$. 
Suppose that $T^{2} \times I \cut L$ is Seifert fibered. 
Then $L$ is isotopic to $T_{2}(p,q)$. 
\end{prop}

Suppose that $L$ is a non-split link in $T^{2} \times I$. 
Let us consider the JSJ decomposition of $T^{2} \times I \cut L$. 
We classify tori in the thickened torus. 

\begin{lem}
\label{lem:2-torus}
Let $T$ be an embedded torus in $T^{2} \times I$. 
Then 
\begin{itemize}
\item $T$ is parallel to the boundary of $T^{2} \times I$, 
\item $T$ bounds a solid torus, or 
\item $T$ bounds the connected sum of a solid torus and a thickened torus. 
\end{itemize}
In the first case, we say that $T$ is a \emph{layering torus}. 
In the last case, $T$ bounds a solid torus or a knotted hole ball. 
\end{lem}
\begin{proof}
Fix a Seifert fibration of $T^{2} \times I$, which has no singular fibers. 
If $T$ is incompressible, 
then $T$ is isotopic to be vertical (a union of regular fibers) 
or horizontal (transverse to all fibers) in the Seifert fibration 
by \cite[Proposition 1.11]{Hatcher07}. 
Since $T$ is closed, $T$ is vertical. 
Hence $T$ is parallel to the boundary of $T^{2} \times I$. 
If $T$ is compressible, 
then $T$ bounds a solid torus or a knotted hole ball 
in the same manner as Lemma~\ref{lem:1-torus}. 
\end{proof}

We say that $L$ is a \emph{layered link} 
if $T^{2} \times I \cut L$ has a layering essential torus. 
We say that $L$ is \emph{satellite link} 
if $T^{2} \times I \cut L$ has an essential torus 
which is not layering. 
Then the following classification of links in the solid torus 
clearly holds by Proposition~\ref{prop:2-seifert}. 

\begin{prop}
\label{prop:2-class}
Suppose that $L$ is a non-split link in $T^{2} \times I$. 
Then $L$ is one of the following: 
\begin{itemize}
\item $T_{2}(p,q)$, 
\item a hyperbolic link, 
\item a satellite link, or 
\item a layered link. 
\end{itemize}
\end{prop}

If $L$ is a layered link in $T^{2} \times I$, 
then $L$ splits to non-layered links along layering incompressible tori. 
Note that $T_{2}(p,q)$ with $\gcd(p,q) > 1$ is also satellite and layered. 
The following lemma is proven 
in the same manner as Lemma~\ref{lem:1-outermost}. 

\begin{lem}
\label{lem:2-outermost}
Let $L$ be a non-split and non-layered link in $T^{2} \times I$. 
Suppose that $M$ is the JSJ piece of $T^{2} \times I \cut L$ which contains a component of $T^{2} \times \partial I$, 
called the outermost JSJ piece. 
Then $M$ is homeomorphic to the complement of a link in $T^{2} \times I$. 
Moreover, this homeomorphism can be taken to fix a neighborhood of $T^{2} \times \partial I$. 
\end{lem}

The classification in \cite[Proposition 3.2]{Budney06} implies the following classification. 

\begin{prop}
\label{prop:2-jsj}
Suppose that $L$ is a non-split link in $T^{2} \times I$. 
Let us consider the JSJ decomposition of $T^{2} \times I \cut L$. 
If a JSJ piece $M$ is Seifert fibered, 
then $M$ is homeomorphic to one of the following: 
\begin{itemize}
\item $M(0, n; \alpha_{1} / \beta_{1}, \alpha_{2} / \beta_{2})$ 
	for $n \geq 1$ and $\alpha_{1} \beta_{2} - \alpha_{2} \beta_{1} = \pm 1$, 
\item $M(0, n; \alpha / \beta)$ for $n \geq 2$, or 
\item $M(0, n; )$ for $n \geq 3$. 
\end{itemize}
The Seifert fibration of $M$ is unique up to isotopy. 
\end{prop}

\subsection{Links in the 3-torus}

In this subsection, we consider links in the 3-torus $T^{3}$. 
We first define elementary links in the 3-torus. 
Let $T^{3} = \bbR^{3} / \bbZ^{3}$. 
For $(p,q,r) \in \bbZ^{3} \setminus \{ (0,0,0) \}$, define 
\[
T_{3}(p,q,r) = \{ (p^{\prime} t, q^{\prime} t, r^{\prime} t + \frac{k}{\gcd(p,q)}) \in T^{3} \mid t \in [0,1], k = 0, \dots, d-1 \}, 
\]
where $d = \gcd(p,q,r)$, $p^{\prime} = p/d$, $q^{\prime} = q/d$, and $r^{\prime} = r/d$. 
The link $T_{3}(p,q,r)$ consists of $d$ slided copies of $T_{3}(p^{\prime}, q^{\prime}, r^{\prime})$. 
Since the links are not oriented, 
the links $T_{3}(p,q,r)$ and $T_{3}(-p,-q,-r)$ are identical. 
We show that the other pairs in $T_{3}(p,q,r)$ are not isotopic.

\begin{lem}
\label{lem:3-t-link}
Suppose that $p_{0}, p_{1} \geq 0$. 
If $T_{3}(p_{0}, q_{0}, r_{0})$ and $T_{3}(p_{1}, q_{1}, r_{1})$ are isotopic, 
then $p_{0} = p_{1}$, $q_{0} = q_{1}$, and $r_{0} = r_{1}$. 
\end{lem}
\begin{proof}
For $i=0,1$, let $d_{i} = \gcd(p_{i}, q_{i}, r_{i})$, $p_{i}^{\prime} = p_{i} / d_{i}$, $q_{i}^{\prime} = q_{i} / d_{i}$, and $r_{i}^{\prime} = r_{i} / d_{i}$. 
Since the number of component of the link $T_{3}(p_{i}, q_{i}, r_{i})$ is $d_{i}$, 
we have $d_{0} = d_{1}$. 
Since any component $T_{3}(p_{i}, q_{i}, r_{i})$ represents 
$\pm (p_{i}^{\prime}, q_{i}^{\prime}, r_{i}^{\prime}) \in \bbZ^{3} = \pi_{1}(X)$, 
we have $p_{0}^{\prime} = p_{1}^{\prime}$, $q_{0}^{\prime} = q_{1}^{\prime}$, and $r_{0}^{\prime} = r_{1}^{\prime}$. 
Hence $p_{0} = p_{1}$, $q_{0} = q_{1}$, and $r_{0} = r_{1}$. 
\end{proof}

The complement of $T_{3}(p,q,r)$ is homeomorphic to the Seifert fibered manifold $M(1, \gcd(p,q,r);)$. 

\begin{prop}
\label{prop:3-seifert}
Let $L$ be a link in $T^{3}$. 
Suppose that $T^{3} \cut L$ is Seifert fibered. 
Then $L$ is isotopic to $T_{3}(p,q,r)$. 
\end{prop}
\begin{proof}
Assume that 
the meridian of a component $K$ of $L$ is 
a fiber in the Seifert fibration on $T^{3} \cut L$. 
Consider $M = T^{3} \cut (L \setminus K)$. 
Let $\Sigma$ be the base space of the Seifert fibration on $T^{3} \cut L$. 
For an arc in $\Sigma$ whose endpoints are contained 
in the boundary component corresponding to $K$, 
we obtain a sphere in $M$ which is the union of the lift of the arc in $T^{3} \cut L$ 
and two disks in a neighborhood of $K$. 
If $\Sigma$ is non-orientable, 
then an orientation-reversing arc induces a non-separating sphere in $M$. 
If $\Sigma$ has a positive genus, 
then a non-separating arc induces a non-separating sphere in $M$, 
which is impossible. 
If there is a singular fiber, 
then an arc in a neighborhood of an arc joining the boundary and the singular point 
induces a sphere in $M$ bounding a lens space minus a ball. 
Since the 3-torus is irreducible, these are impossible. 
Hence $\Sigma$ is a non-singular punctured sphere. 
and $T^{3} \cut L$ is homeomorphic to $\Sigma \times S^{1}$, 
and $M$ is homeomorphic to the complement of an unlink. 
Then $T^{3}$ is the connected sum of some lens spaces or $S^{2} \times S^{1}$'s. 
This contradicts the fact that $T^{3}$ is irreducible. 

Hence the meridian of any component of $L$ is not 
a fiber in the Seifert fibration on $T^{3} \cut L$. 
In this case, the Seifert fibration on $T^{3} \cut L$ extends to $T^{3}$. 
Any Seifert fibration of the 3-torus is isomorphic to 
the product fibering of $M(1,0;)$ 
by \cite[Theorem 2.3]{Hatcher07}. 
Moreover, $L$ consists of regular fibers. 
Therefore $L$ is isotopic to $T_{3}(p,q,r)$. 
\end{proof}

Suppose that $L$ is a non-split link in $T^{3}$. 
Let us consider the JSJ decomposition of $T^{3} \cut L$. 
We classify tori in the 3-torus. 

\begin{lem}
\label{lem:3-torus}
Let $T$ be an embedded torus in $T^{3}$. 
Then 
\begin{itemize}
\item $T^{3} \cut T$ is a thickened torus, 
\item $T$ bounds a solid torus, or 
\item $T$ bounds the connected sum of a solid torus and a 3-torus. 
\end{itemize}
In the first case, we say that $T$ is a \emph{layering torus}. 
In the last case, $T$ bounds a solid torus or a knotted hole ball. 
\end{lem}
\begin{proof}
Fix a Seifert fibration of $T^{3}$, which has no singular fibers. 
If $T$ is incompressible, 
then $T$ is isotopic to be vertical or horizontal in the Seifert fibration 
by \cite[Proposition 1.11]{Hatcher07}. 
In both cases, $T^{3} \cut T$ is a thickened torus. 
If $T$ is compressible, 
then $T$ bounds a solid torus or a knotted hole ball 
in the same manner as Lemma~\ref{lem:1-torus}. 
\end{proof}

We say that $L$ is a \emph{layered link} 
if $T^{3} \cut L$ has a layering essential torus. 
We say that $L$ is a \emph{satellite link} 
if $T^{3} \cut L$ has an essential torus 
which is not layering. 
Then the following classification of links in the solid torus 
clearly holds by Proposition~\ref{prop:3-seifert}. 

\begin{prop}
\label{prop:3-class}
Suppose that $L$ is a non-split link in $T^{3}$. 
Then $L$ is one of the following: 
\begin{itemize}
\item $T_{3}(p,q,r)$, 
\item a hyperbolic link, 
\item a satellite link, or 
\item a layered link. 
\end{itemize}
\end{prop}

If $L$ is a layered link in $T^{3}$, 
then $L$ splits to non-layered links along layering incompressible tori. 
The following lemma is proven 
in the same manner as Lemma~\ref{lem:1-outermost}. 

\begin{lem}
\label{lem:3-outermost}
Let $L$ be a non-split and non-layered link in $T^{3}$. 
Suppose that $M$ is the JSJ piece of $T^{3} \cut L$ 
such that $\overline{T^{3} \setminus M}$ is a disjoint union of solid tori and knotted hole balls, 
called the outermost JSJ piece. 
Then $M$ is homeomorphic to the complement of a link in the 3-torus. 
\end{lem}

\begin{prop}
\label{prop:3-jsj}
Suppose that $L$ is a non-split link in $T^{3}$. 
Let us consider the JSJ decomposition of $T^{3} \cut L$. 
If a JSJ piece $M$ is Seifert fibered, 
then $M$ is homeomorphic to one of the following: 
\begin{itemize}
\item $M(0, n; \alpha_{1} / \beta_{1}, \alpha_{2} / \beta_{2})$ 
	for $n \geq 1$ and $\alpha_{1} \beta_{2} - \alpha_{2} \beta_{1} = \pm 1$, 
\item $M(0, n; \alpha / \beta)$ for $n \geq 2$, 
\item $M(0, n; )$ for $n \geq 3$, or 
\item $M(1, n; )$ for $n \geq 1$. 
\end{itemize}
The Seifert fibration of $M$ is unique up to isotopy. 
\end{prop}
\begin{proof}
If $L$ is non-layered and the outermost JSJ piece $M$ is Seifert fibered, 
then $M$ is homeomorphic to $M(1, n; )$ 
by Proposition~\ref{prop:3-seifert} and Lemma~\ref{lem:3-outermost}. 
The other Seifert fibered pieces 
are those in Proposition~\ref{prop:2-jsj}. 
\end{proof}

\section{Cover-isotopy of links}
\label{section:isotopy}

We say that two links $L_{0}$ and $L_{1}$ in a 3-manifold $X$ is \emph{cover-isotopic} 
if there is a finite covering map $P \colon \widetilde{X} \to X$ 
such that the lifts (i.e. the preimages) $\widetilde{L_{0}} = P^{-1}(L_{0})$ and $\widetilde{L_{1}} = P^{-1}(L_{1})$ are isotopic in $\widetilde{X}$. 
It is quite non-trivial whether cover-isotopy implies isotopy. 

For instance, Manfredi \cite{Manfredi14} gave examples of non-isotopic links in a lens space having isotopic lifts in the 3-sphere. 
The last author \cite{Yoshida25a} showed there are infinitely many triples of non-isotopic hyperbolic links in the lens space $L(4,1)$ having isotopic lifts in the 3-sphere. 
It is also worth pointing out that cover-isotopy does not imply isotopy for links in the twisted $I$-bundle over the Klein bottle $S^{1} \tilde{\times} S^{1} \tilde{\times} I$ (also known as the thickened Klein bottle) in general. 
By arguments similar to ours, it can be proved that cover-isotopy implies isotopy for the links in the projective 3-space \cite{Yoshida25b}.

We first show that cover-isotopy implies isotopy 
for hyperbolic links in the solid torus, the thickened torus, or the 3-torus. 
We need to compare group actions on a hyperbolic 3-manifold. 
For a group $G$ and a topological space $M$, 
let $\alpha \colon G \to \Aut(M)$ be a homomorphism, 
called a \emph{(left) action} of $G$ on $M$, 
where $\Aut(M)$ is the group consisting of the self-homeomorphisms of $M$. 
An action is \emph{free} if $\alpha (g)$ has no fixed points for each non-trivial element $g \in G$. 
A homeomorphism $f \colon M \to M^{\prime}$ 
induces the action $f_{*}(\alpha)$ on $M^{\prime}$ 
defined by $f_{*}(\alpha) (g) = f \circ \alpha (g) \circ f^{-1} \in \Aut(M^{\prime})$ for $g \in G$. 
An \emph{isometric action} of $G$ on a hyperbolic 3-manifold $M$ is 
a homomorphism $\alpha \colon G \to \Isom(M)$, 
where $\Isom(M)$ is the group consisting of the self-isometries of $M$. 
An isometric action on a hyperbolic 3-manifold induces an action on the cusps 
which is isometric with respect to the Euclidean structure. 
If this action is free, the group acts on the tori by translations and interchanging components. 

To consider actions on a torus, 
we use the notion of rotation set as follows. 
Let $T^{m} = \bbR^{m} / \bbZ^{m}$ be an $m$-torus. 
Consider the group $\Aut_{0}(T^{m})$ consisting of the self-homeomorphisms of $T^{m}$ isotopic to the identity. 
Suppose that $h \in \Aut_{0}(T^{m})$, 
and a homeomorphism $\tilde{h} \colon \bbR^{m} \to \bbR^{m}$ is a lift of $h$. 
In other words, it holds that $h \circ P = P \circ \tilde{h}$ 
for the natural projection $P \colon \bbR^{m} \to \bbR^{m} / \bbZ^{m}$. 
The lift is unique up to the translations by $\bbZ^{m}$. 
Misiurewicz and Ziemian \cite{MZ89} defined the \emph{rotation set} of $\tilde{h}$ by 
\[
\rho (\tilde{h}) = \bigcap_{k=1}^{\infty} \overline{\left( \bigcup_{n=k}^{\infty} \left\{ \frac{\tilde{h}^{n}(x) - x}{n} \mid x \in \bbR^{2} \right\} \right)} \subset \bbR^{m}, 
\]
which is an analogue of the rotational number of a self-homeomorphism of $S^{1}$. 
If $\tilde{f} \colon \bbR^{m} \to \bbR^{m}$ is a lift of $f \in \Aut_{0}(T^{m})$, 
then $\rho (\tilde{f} \circ \tilde{h} \circ \tilde{f}^{-1}) = \rho (\tilde{h})$. 

\begin{lem}
\label{lem:translation}
Let $T^{m} = \bbR^{m} / \bbZ^{m}$ be an $m$-torus with a Euclidean structure. 
Let $h_{0}, h_{1} \colon T^{m} \to T^{m}$ be isometries isotopic to the identity. 
Suppose that there is a homeomorphism $f \colon T^{m} \to T^{m}$ isotopic to the identity satisfying $f \circ h_{0} \circ f^{-1} = h_{1}$. 
Then $h_{0} = h_{1}$. 
\end{lem}
\begin{proof}
If $h \colon T^{m} \to T^{m}$ is an isometry isotopic to the identity, 
then $h \in \Aut_{0}(T^{m})$ is a translation by an element $v \in \bbR^{m} / \bbZ^{m}$. 
The rotation set $\rho (\tilde{h})$ is determined as a single point $\tilde{v} \in \bbR^{m}$ modulo $\bbZ^{m}$ that is a lift of $v$. 
Since there are lifts $\tilde{h}_{0}, \tilde{h}_{1}, \tilde{f} \colon \bbR^{m} \to \bbR^{m}$ of $h_{0}, h_{1}, f \colon T^{m} \to T^{m}$ such that $\tilde{f} \circ \tilde{h}_{0} \circ \tilde{f}^{-1} = \tilde{h}_{1}$, 
we have $\rho (\tilde{h}_{0}) \equiv \rho (\tilde{h}_{1})$ modulo $\bbZ^{m}$. 
Hence the translations $h_{0}$ and $h_{1}$ coincide. 
\end{proof}

We state a sufficient condition for cover-isotopy of hyperbolic links to imply isotopy in a general setting. 
Suppose that $X$ is a 3-manifold with (possibly empty) boundary consisting of tori. 
Let $L_{0}$ and $L_{1}$ be hyperbolic links in $X$. 
Let $P \colon \widetilde{X} \to X$ be a finite regular covering map. 
Suppose that the lifts $\widetilde{L_{0}} = P^{-1}(L_{0})$ and $\widetilde{L_{1}} = P^{-1}(L_{1})$ are isotopic in $\widetilde{X}$. 
Then the interior of the complement $M_{i} = X \cut L_{i}$ admits a hyperbolic metric for $i=0,1$. 
Its cover $\widetilde{M_{i}} = \widetilde{X} \cut \widetilde{L_{i}}$ 
is endowed with the lifted hyperbolic metric. 
The isotopy between $\widetilde{L_{0}}$ and $\widetilde{L_{1}}$ induces a homeomorphism $f \colon \widetilde{M_{0}} \to \widetilde{M_{1}}$. 
The Mostow rigidity theorem implies that $f$ is isotopic to an isometry $\tilde{\iota} \colon \widetilde{M_{0}} \to \widetilde{M_{1}}$. 
The finite group $G = \pi_{1}(X) / \pi_{1}(\widetilde{X})$ acts on $\widetilde{X}$ by the deck transformations. 
This action induces free isometric actions of $G$ on the hyperbolic 3-manifolds $\widetilde{M_{0}}$ and $\widetilde{M_{1}}$, denoted by $\alpha_{0}$ and $\alpha_{1}$. 

\begin{lem}
\label{lem:hyp-isotopy}
Under the above assumptions, 
if the actions $\tilde{\iota}_{*}(\alpha_{0})$ and $\alpha_{1}$ on $\widetilde{M_{1}}$ coincide, 
then the pairs $(X, L_{0})$ and $(X, L_{1})$ are homeomorphic. 
Furthermore, suppose that if a self-homeomorphism $f$ of $X$ lifts to a self-homeomorphism $\tilde{f}$ of $\widetilde{X}$ isotopic to the identity, then $f$ is isotopic to the identity. 
In this case, $L_{0}$ and $L_{1}$ are isotopic in $X$. 
\end{lem}
\begin{proof}
Since $\tilde{\iota}_{*}(\alpha_{0}) = \alpha_{1}$, 
the isometry $\tilde{\iota} \colon \widetilde{M_{0}} \to \widetilde{M_{1}}$ is $G$-equivariant. 
Hence $\tilde{\iota}$ induces an isometry $\iota \colon M_{0} \to M_{1}$. 
Moreover, $\iota$ preserves the meridians. 
Hence $\iota$ can be extended to a homeomorphism from $(X, L_{0})$ to $(X, L_{1})$. 

The homeomorphism $\iota \colon (X, L_{0}) \to (X, L_{1})$ lifts to a homeomorphism $\tilde{\iota} \colon (\widetilde{X}, \widetilde{L_{0}}) \to (\widetilde{X}, \widetilde{L_{1}})$ isotopic to the identity of $\widetilde{X}$. 
If $\iota$ is isotopic to the identity of $X$, 
then $L_{0}$ and $L_{1}$ are isotopic in $X$. 
\end{proof}

To use Lemma~\ref{lem:hyp-isotopy} effectively, 
we need the following lemma, 
which concerns the Birman--Hilden property 
(see \cite{MW21}). 

\begin{lem}
\label{lem:isotopy-lift}
Suppose that $X$ is $S^{1} \times D^{2}$, $T^{2} \times I$, or $T^{3}$. 
Let $P \colon \widetilde{X} \to X$ be a finite covering map. 
Suppose that a self-homeomorphism $f$ of $X$ lifts to a self-homeomorphism $\tilde{f}$ of $\widetilde{X}$ isotopic to the identity. 
Then $f$ is isotopic to the identity. 
\end{lem}
\begin{proof}
Recall that $\Aut(X)$ is the group consisting of self-homeomorphisms of $X$, and $\Aut_{0}(X)$ is the subgroup of $\Aut(X)$ consisting of self-homeomorphisms of $X$ isotopic to the identity. 
The mapping class group of $X$ is defined as the quotient group $\pi_{0}(\Aut(X)) = \Aut(X) / \Aut_{0}(X)$. 
The action on the fundamental group induces the natural homomorphism from $\pi_{0}(\Aut(X))$ to the outer automorphism group $\Out(\pi_{1}(X))$. 
Moreover, $\Out(\bbZ^{n}) \cong \mathrm{GL}(n, \bbZ)$. 
Due to Hatcher \cite{Hatcher76}, the following isomorphisms are known to hold: 
\begin{align*}
\pi_{0}(\Aut(S^{1} \times D^{2})) & \cong (\bbZ / 2\bbZ \times \bbZ) \rtimes \bbZ / 2\bbZ, \\
\pi_{0}(\Aut(T^{2} \times I)) & \cong \mathrm{GL}(2, \bbZ) \times \bbZ / 2\bbZ, \\
\pi_{0}(\Aut(T^{3})) & \cong \mathrm{GL}(3, \bbZ). 
\end{align*}
Note that we allow orientation-reversing self-homeomorphisms. 
In these cases, the homomorphism $\pi_{0}(\Aut(X)) \to \Out(\pi_{1}(X))$ is surjective. 
The first $\bbZ / 2\bbZ$-subgroup of $\pi_{0}(\Aut(S^{1} \times D^{2}))$ is $\Out(\pi_{1}(S^{1} \times D^{2}))$. The $\bbZ$-subgroup of $\pi_{0}(\Aut(S^{1} \times D^{2}))$ is generated by the Dehn twist. The last $\bbZ / 2\bbZ$-subgroup of $\pi_{0}(\Aut(S^{1} \times D^{2}))$ is generated by a reflection preserving each $\{ t \} \times D^{2}$. The $\bbZ / 2\bbZ$-subgroup of $\pi_{0}(\Aut(T^{2} \times I))$ is generated by the reflection about $T^{2} \times \{ 0 \}$. 

Since $\pi_{1}(X)$ is abelian, the isomorphism $f_{*} \colon \pi_{1}(X) \to \pi_{1}(X)$ does not depend on the choice of a base point. 
A self-homeomorphism $f \in \Aut(X)$ lifts to $\tilde{f} \in \Aut(\widetilde{X})$ 
if and only if the two subgroups $f_{*} \circ P_{*}(\pi_{1}(\widetilde{X}))$ and $P_{*}(\pi_{1}(\widetilde{X}))$ of $\pi_{1}(X)$ coincide. 
In other words, the subgroup $P_{*}(\pi_{1}(\widetilde{X}))$ of $\pi_{1}(X)$ is invariant under $f_{*}$. 
Then $\tilde{f}_{*} \colon \pi_{1}(\widetilde{X}) \to \pi_{1}(\widetilde{X})$ is the restriction of $f_{*}$. 

Suppose that $\tilde{f} \in \Aut_{0}(\widetilde{X})$. 
Then $\tilde{f}_{*} \colon \pi_{1}(\widetilde{X}) \to \pi_{1}(\widetilde{X})$ is the identity. 
Hence $f_{*} \colon \pi_{1}(X) \to \pi_{1}(X)$ is also the identity. 
In other words, $[f] \in \pi_{0}(\Aut(X))$ is mapped to the identity in $\Out(\pi_{1}(X))$. 
Since $\tilde{f}$ is orientation-preserving, so is $f$. 
Hence $f \in \Aut_{0}(X)$ in the cases that $X$ is $T^{2} \times I$ or $T^{3}$. 
In the case that $X$ is $S^{1} \times D^{2}$, we need to pay attention to the $\bbZ$-subgroup of $\pi_{0}(\Aut(S^{1} \times D^{2}))$. 
If its element for $[f] \in \pi_{0}(\Aut(X))$ is $a$, 
then that for $[\tilde{f}] \in \pi_{0}(\Aut(\widetilde{X}))$ is the product of $a$ and the degree of the covering map $P \colon \widetilde{X} \to X$, which is equal to zero. 
Hence $a = 0$. 
Therefore $f \in \Aut_{0}(X)$. 
\end{proof}

Note that if $X$ is $S^{1} \times S^{2}$, the conclusion in Lemma~\ref{lem:isotopy-lift} does not hold. 
Indeed, $\pi_{0}(\Aut(S^{1} \times S^{2})) \cong \bbZ / 2\bbZ \times \bbZ / 2\bbZ \times \bbZ / 2\bbZ$. 
Its $\bbZ / 2\bbZ$-subgroup instead of the $\bbZ$-subgroup of $\pi_{0}(\Aut(S^{1} \times D^{2}))$ is due to $\pi_{1}(\mathrm{SO}(3)) \cong \bbZ / 2\bbZ$ (cf. the belt trick). 
Hence a self-homeomorphism of $S^{1} \times S^{2}$ corresponding to $(0,1,0) \in \bbZ / 2\bbZ \times \bbZ / 2\bbZ \times \bbZ / 2\bbZ$ lifts to a self-homeomorphism of the double cover isotopic to the identity.

\begin{thm}
\label{thm:hyp-isotopy}
Suppose that $X$ is $S^{1} \times D^{2}$, $T^{2} \times I$, or $T^{3}$. 
Let $L_{0}$ and $L_{1}$ be hyperbolic links in $X$. 
Let $P \colon \widetilde{X} \to X$ be a finite covering map. 
Suppose that the lifts $\widetilde{L_{0}} = P^{-1}(L_{0})$ and $\widetilde{L_{1}} = P^{-1}(L_{1})$ are isotopic in $\widetilde{X}$. 
Then $L_{0}$ and $L_{1}$ are isotopic in $X$. 
\end{thm}
\begin{proof}
Since the fundamental group $\pi_{1}(X)$ is abelian, 
the covering map $P$ is regular. 
We already have an isometry $\tilde{\iota} \colon \widetilde{M_{0}} \to \widetilde{M_{1}}$ and a free isometric action $\alpha_{i}$ of the finite group $G = \pi_{1}(X) / \pi_{1}(\widetilde{X})$ on the hyperbolic 3-manifold $\widetilde{M_{i}}$ for each $i = 0,1$. 
By Lemmas \ref{lem:hyp-isotopy} and \ref{lem:isotopy-lift}, 
it is sufficient to show that $\tilde{\iota}_{*}(\alpha_{0}) = \alpha_{1}$. 

We first suppose that $X$ is $S^{1} \times D^{2}$ or $T^{2} \times I$. 
Then we may assume that the isotopy between $\widetilde{L_{0}}$ and $\widetilde{L_{1}}$ fixes a neighborhood $H$ of a component of $\partial \widetilde{X}$, 
and the hyperbolic metrics of $\widetilde{M_{0}}$ and $\widetilde{M_{1}}$ coincide in the horocusp $H$. 
The isotopy from the homeomorphism $f \colon \widetilde{M_{0}} \to \widetilde{M_{1}}$ to the isometry $\tilde{\iota} \colon \widetilde{M_{0}} \to \widetilde{M_{1}}$ can be taken to fix $H$ setwise. 
Moreover, $f_{*}(\alpha_{0}) = \alpha_{1}$ on $H$. 
The isometric action $\tilde{\iota}_{*}(\alpha_{0})$ on the Euclidean torus $\partial H$ is conjugate to $f_{*}(\alpha_{0}) = \alpha_{1}$ by 
$\tilde{\iota} \circ f^{-1} \in \Aut_{0}(\partial H)$. Lemma~\ref{lem:translation} implies that the isometries $\tilde{\iota}_{*}(\alpha_{0})(g)$ and $\alpha_{1}(g)$ for each $g \in G$ coincide on $\partial H$. 
Hence they coincide on $H$, and so on whole $\widetilde{M_{1}}$. 
Therefore $\tilde{\iota}_{*}(\alpha_{0}) = \alpha_{1}$.

We next suppose that $X$ is $T^{3}$. 
We consider the two subgroups $G_{0} = \tilde{\iota}_{*}(\alpha_{0})(G)$ and $G_{1} = \alpha_{1}(G)$ of the finite group $\Isom(\widetilde{M_{1}})$. 
Let $\overline{G}$ denote the subgroup of $\Isom(\widetilde{M_{1}})$ generated by $G_{0}$ and $G_{1}$. 
Since the elements of $G_{0}$ and $G_{1}$ preserves the meridians, 
so do the elements of $\overline{G}$. 
Hence the action of the finite group $\overline{G}$ on $\widetilde{M_{1}}$ can be extended to a smooth action on $\widetilde{X}$. 
Note that this action might not be free. 
The 3-manifold or 3-orbifold $\widetilde{X}/ \overline{G}$ is a quotient of a 3-torus. 
Due to Meeks and Scott \cite{MS86}, $\widetilde{X}/ \overline{G}$ admits a Euclidean structure 
(This assertion also follows from the geometrization of 3-manifolds and 3-orbifolds \cite{BLP05, CHK00, Thurston82}). 
Hence $\widetilde{X}$ admits a Euclidean structure such that the action of $\overline{G}$ on $\widetilde{X}$ is isometric. 
The actions $\tilde{\iota}_{*}(\alpha_{0})$ and $\alpha_{1}$ 
of $G$ on the Euclidean torus $\widetilde{X}$ are isometric. 
Since they are conjugate by a self-homeomorphism of $\widetilde{X}$ isotopic to the identity, 
we have $\tilde{\iota}_{*}(\alpha_{0}) = \alpha_{1}$ by Lemma~\ref{lem:translation}. 
\end{proof}

We show that cover-isotopy implies isotopy for links in the solid torus. 

\begin{thm}
\label{thm:1-isotopy}
Let $L_{0}$ and $L_{1}$ be links in $X = S^{1} \times D^{2}$. 
Let $P \colon \widetilde{X} \to X$ be a finite covering map. 
Suppose that the lifts $\widetilde{L_{0}} = P^{-1}(L_{0})$ and $\widetilde{L_{1}} = P^{-1}(L_{1})$ are isotopic in $\widetilde{X}$. 
Then $L_{0}$ and $L_{1}$ are isotopic in $X$. 
\end{thm}

We give an example to note. 
Let $L_{0} = T_{0}(2,0)$ and $L_{1} = T_{0}(2,1)$. 
Since the numbers of their components are different, 
there is no self-homeomorphism of $X$ 
that maps $L_{0}$ to $L_{1}$. 
Suppose that $\widetilde{X}$ is the double cover of $X$. 
Then there is a self-homeomorphism of $\widetilde{X}$ 
that maps $\widetilde{L_{0}} = T_{0}(2,0)$ to $\widetilde{L_{1}} = T_{0}(2,2)$ as shown in Figure~\ref{fig5}. 
However, they are not isotopic in fixed $\widetilde{X}$. 
Hence this is not a counterexample of Theorem~\ref{thm:1-isotopy}. 
There are such examples of hyperbolic links as shown in Figure~\ref{fig6}.

\begin{figure}[ht]
\centerline{\includegraphics[width=5in]{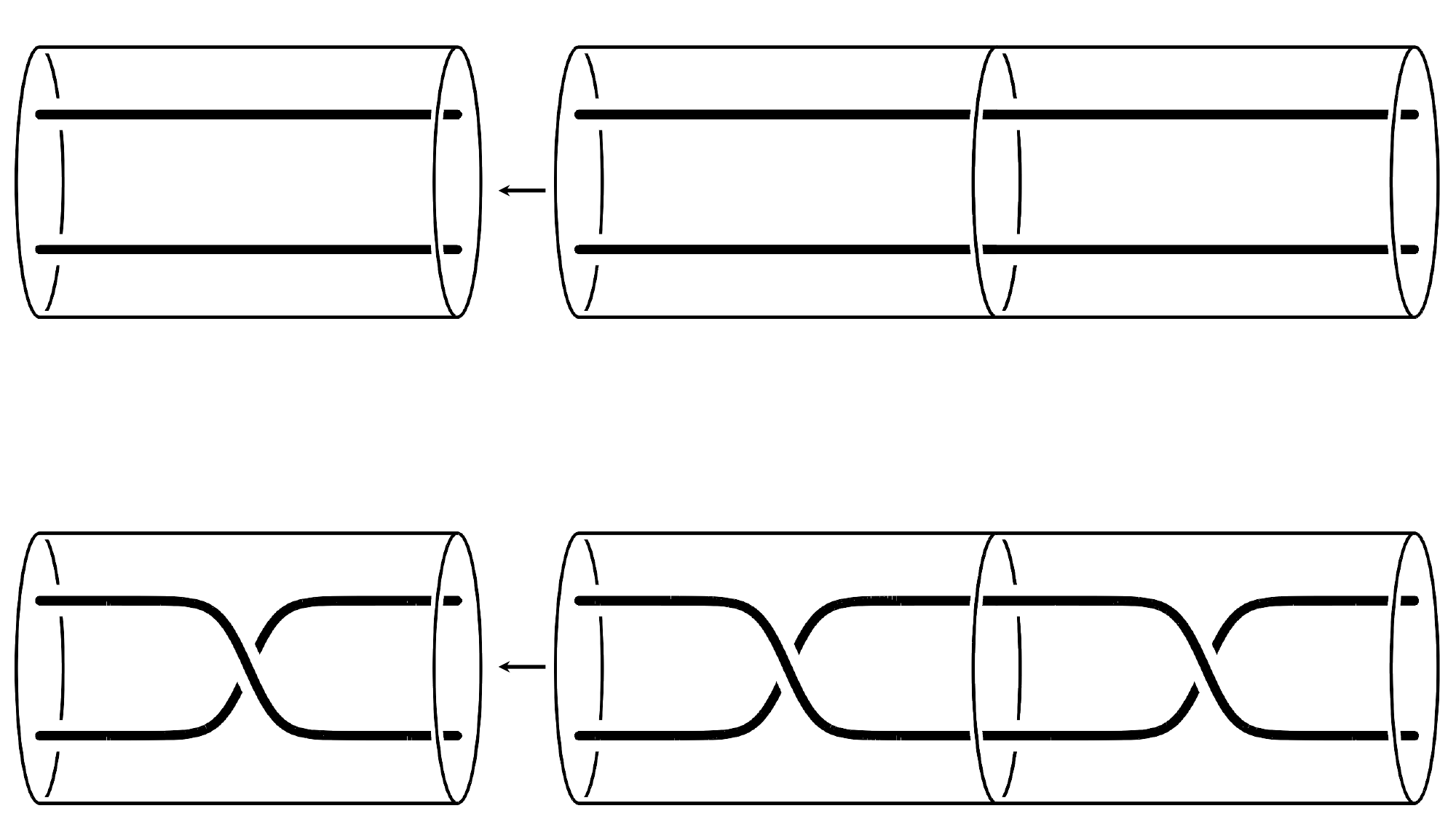}}
\vspace*{8pt}
\caption{\label{fig5} 
Seifert fibered inks in the solid torus with homeomorphic lifts.}
\end{figure}

\begin{figure}[ht]
\centerline{\includegraphics[width=5in]{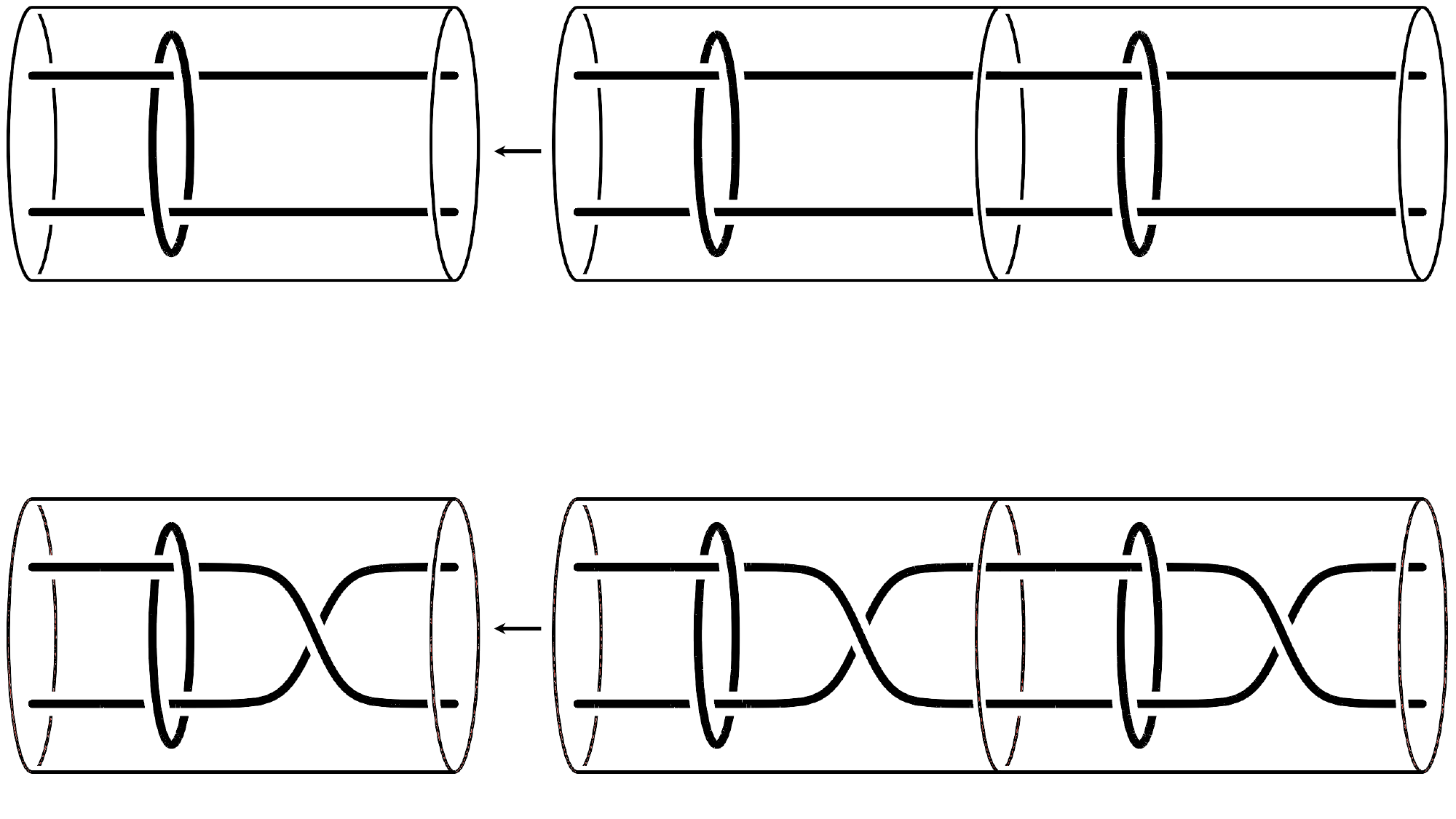}}
\vspace*{8pt}
\caption{\label{fig6} 
Hyperbolic links in the solid torus with homeomorphic lifts.}
\end{figure}

Suppose that $n$ is the degree of the covering map $P \colon \widetilde{X} \to X$. 

\begin{lem}
\label{lem:1-seifert-isotopy}
Theorem~\ref{thm:1-isotopy} holds 
in the case that $L_{0}$ and $L_{1}$ are Seifert fibered. 
\end{lem}
\begin{proof}
Proposition~\ref{prop:1-seifert} implies that 
$L_{0}$ is isotopic to $T_{l_{0}}(p_{0}, q_{0})$, 
and $L_{1}$ is isotopic to $T_{l_{1}}(p_{1}, q_{1})$. 
We may assume that $p_{0}, p_{1} \geq 0$, and $p_{i}$ does not divide $q_{i}$ if $l_{i} =0$. 
Then $\widetilde{L_{0}}$ is isotopic to $T_{l_{0}}(p_{0}, nq_{0})$, 
and $\widetilde{L_{1}}$ is isotopic to $T_{l_{1}}(p_{1}, nq_{1})$. 
Hence $T_{l_{0}}(p_{0}, nq_{0})$ and $T_{l_{1}}(p_{1}, nq_{1})$ are isotopic. 
If $l_{0} = 0$ and $p_{0}$ divides $nq_{0}$, 
then $T_{0}(p_{0}, nq_{0})$ is isotopic to $T_{1}(p_{0}-1, n(p_{0}-1)q_{0}/p_{0})$. 
In this case, since $p_{0}$ does not divide $(p_{0}-1)q_{0}$, 
$T_{1}(p_{0}-1, n(p_{0}-1)q_{0}/p_{0})$ is not the $n$-sheeted cover of $T_{1}(p_{1}, q_{1})$. 
Hence, in any case, we have $l_{0} = l_{1}$, $p_{0} = p_{1}$, and $q_{0} = q_{1}$ by Lemma~\ref{lem:1-t-link}. 
Theorefore $L_{0}$ and $L_{1}$ are isotopic. 
\end{proof}

\begin{proof}[Proof of Theorem~\ref{thm:1-isotopy}]
We first consider the case that $L_{0}$ is a split link. 
In other words, $X \cut L_{0}$ is reducible. 
Suppose that 
the link $L_{0}$ maximally splits to $L_{0}^{\prime}$, $L_{0,1}$, \dots, $L_{0,m}$, 
where $L_{0}^{\prime}$ is a (possibly empty) non-split link. 
Then $\widetilde{L_{0}}$ maximally splits to 
links $\widetilde{L_{0}^{\prime}}$ and $L_{0,j}^{k}$ for $1 \leq j \leq m$ and $1 \leq k \leq n$, 
where $\widetilde{L_{0}^{\prime}}$ is the lift of $L_{0}^{\prime}$, 
and $L_{0,j}^{k}$ is a copy of $L_{0,j}$. 
The link $\widetilde{L_{0}^{\prime}}$ is a non-split link. 
In the same manner, 
the link $L_{1}$ maximally splits to $L_{1}^{\prime}$, $L_{1,1}$, \dots, $L_{1,m^{\prime}}$, 
and the lift $\widetilde{L_{1}}$ maximally splits to 
$\widetilde{L_{1}^{\prime}}$ and $L_{1,j}^{k}$ for $1 \leq j \leq m^{\prime}$ and $1 \leq k \leq n$. 
Since $\widetilde{L_{0}}$ and $\widetilde{L_{1}}$ are isotopic in $\widetilde{X}$, 
Lemma~\ref{lem:prime} implies that 
$\widetilde{L_{0}^{\prime}}$ and $\widetilde{L_{1}^{\prime}}$ are isotopic, 
$m = m^{\prime}$, and $L_{0,j}^{k}$ and $L_{1,j}^{k}$ are isotopic 
after possibly permuting the indices. 
Hence $L_{0,j}$ and $L_{1,j}$ are isotopic. 
If $L_{0}^{\prime}$ and $L_{1}^{\prime}$ are isotopic, 
then $L_{0}$ and $L_{1}$ are isotopic. 
Therefore we may assume that $L_{0}$ and $L_{1}$ are non-split links. 

We consider the JSJ decomposition of the irreducible 3-manifolds 
$X \cut L_{0}$ and $X \cut L_{1}$. 
The JSJ decomposition of 
$\widetilde{X} \cut \widetilde{L_{0}} \simeq \widetilde{X} \cut \widetilde{L_{1}}$ 
is their lifts by Lemma~\ref{lem:jsj-cover}. 
We show the assertion 
by induction on the number $N$ of JSJ pieces of $\widetilde{X} \cut \widetilde{L_{0}}$. 
We first suppose that $\widetilde{X} \cut \widetilde{L_{0}} 
(\simeq \widetilde{X} \cut \widetilde{L_{1}})$ has no essential tori. 
Then the link $\widetilde{L_{0}}$ is either Seifert fibered or hyperbolic, 
and the links $L_{0}$ and $L_{1}$ are in the same class. 
In this case, the assertion holds 
by Theorem~\ref{thm:hyp-isotopy} and Lemma~\ref{lem:1-seifert-isotopy}. 

Suppose that $N > 1$ 
and the assertion holds for any degree $n$ of the covering map 
if the number of JSJ pieces is less than $N$. 
Let $M_{0}$ and $M_{1}$ respectively denote 
the outermost JSJ pieces of $X \cut L_{0}$ and $X \cut L_{1}$. 
Their preimages $\widetilde{M_{0}}$ and $\widetilde{M_{1}}$ in $\widetilde{X}$ are isotopic in $\widetilde{X}$. 
The 3-manifolds $M_{0}$ and $M_{1}$ can be re-embedded as the complements of links in $X$ 
by Lemma~\ref{lem:1-outermost}. 
Here a knotted hole ball for $M_{i}$ is lifted to its disjoint $n$ copies in $\widetilde{X}$. 
The 3-manifolds $\widetilde{M_{0}}$ and $\widetilde{M_{1}}$ are equivariantly re-embedded. 
Since $M_{0}$ and $M_{1}$ are either Seifert fibered or hyperbolic, 
$M_{0}$ and $M_{1}$ are isotopic in $X$ 
by Theorem~\ref{thm:hyp-isotopy} and Lemma~\ref{lem:1-seifert-isotopy}. 
Note that if $M_{i}$ is Seifert fibered 
and has a boundary component bounding a knotted hole ball, 
then $M_{i}$ is re-embedded as $X \cut T_{1}(0,q)$. 

We show that the isotopy between $M_{0}$ and $M_{1}$ 
extends to an isotopy between $X \cut L_{0}$ and $X \cut L_{1}$. 
Since the knotted hole balls are lifted disjointly, 
the isotopy extends to the knotted hole balls 
similarly to the case for split links. 
The remaining regions are solid tori by Lemma~\ref{lem:1-torus}. 
By the assumption of induction, 
the isotopy in $\widetilde{X}$ induces isotopies between links in these solid tori. 
Therefore $L_{0}$ and $L_{1}$ are isotopic. 
\end{proof}

We show that cover-isotopy implies isotopy for links in the thickened torus. 

\begin{thm}
\label{thm:2-isotopy}
Let $L_{0}$ and $L_{1}$ be links in $X = T^{2} \times I$. 
Let $P \colon \widetilde{X} \to X$ be a finite covering map. 
Suppose that the lifts $\widetilde{L_{0}} = P^{-1}(L_{0})$ and $\widetilde{L_{1}} = P^{-1}(L_{1})$ are isotopic in $\widetilde{X}$. 
Then $L_{0}$ and $L_{1}$ are isotopic in $X$. 
\end{thm}

Suppose that 
the covering space $\widetilde{X}$ corresponds to the subgroup 
of $\pi_{1}(X) = \pi_{1}(\bbR^{2} / \bbZ^{2})= \bbZ^{2}$ 
generated by $\bm{a} = (a_{1}, a_{2})^{\mathsf{T}}$ and $\bm{b} = (b_{1}, b_{2})^{\mathsf{T}}$. 
($(\cdot)^{\mathsf{T}}$ indicates the transpose.) 
The degree of the covering map $P$ is $|\det(\bm{a} \ \bm{b})|$. 

\begin{lem}
\label{lem:2-seifert-isotopy}
Theorem~\ref{thm:2-isotopy} holds 
in the case that $L_{0}$ and $L_{1}$ are Seifert fibered. 
\end{lem}
\begin{proof}
Proposition~\ref{prop:2-seifert} implies that 
$L_{0}$ is isotopic to $T_{2}(p_{0}, q_{0})$, 
and $L_{1}$ is isotopic to $T_{2}(p_{1}, q_{1})$. 
We may assume that $p_{0}, p_{1} \geq 0$. 
We take $(\bm{a}, \bm{b})$ as a basis for $\pi_{1}(\widetilde{X})$. 
Then $\widetilde{L_{i}}$ for $i=0,1$ is isotopic to 
$T_{2}(\tilde{p}_{i}, \tilde{q}_{i})$, 
where 
\[
\tilde{p}_{i} = \det(\bm{x}_{i} \ \bm{b}), \
\tilde{q}_{i} = \det(\bm{a} \ \bm{x}_{i}), \
\bm{x}_{i} = (p_{i}, q_{i})^{\mathsf{T}}. 
\]
Since $\widetilde{L_{0}}$ and $\widetilde{L_{1}}$ are isotopic, 
we have $p_{0} = p_{1}$ and $q_{0} = q_{1}$ 
by Lemma~\ref{lem:2-t-link}. 
Hence $L_{0}$ and $L_{1}$ are isotopic. 
\end{proof}

\begin{proof}[Proof of Theorem~\ref{thm:2-isotopy}]
In the same manner as Theorem~\ref{thm:1-isotopy}, 
we may assume that $L_{0}$ and $L_{1}$ are non-split links. 
The JSJ decomposition of 
$\widetilde{X} \cut \widetilde{L_{0}} \simeq \widetilde{X} \cut \widetilde{L_{1}}$ 
is lifts of those of $X \cut L_{0}$ and $X \cut L_{1}$ 
by Lemma~\ref{lem:jsj-cover}. 
Suppose that $L_{0}$ is decomposed into $L_{0,1}, \dots, L_{0,m}$ along the layering JSJ tori, 
where $L_{0,j}$ are arranged in the order of coordinates for $I$. 
Then $\widetilde{L_{0}}$ is decomposed into their lifts $\widetilde{L_{0,1}}, \dots, \widetilde{L_{0,m}}$ 
along the layering JSJ tori. 
The same applies to $L_{1}$. 
Since the decomposition along the layering JSJ tori is unique up to isotopy, 
$\widetilde{L_{0,j}}$ and $\widetilde{L_{1,j}}$ are isotopic for any $1 \leq j \leq m$. 
Therefore we may assume that $L_{0}$ and $L_{1}$ are not layered by JSJ tori. 

Let $M_{0}$ and $M_{1}$ respectively denote 
the outermost JSJ pieces of $X \cut L_{0}$ and $X \cut L_{1}$ 
as in Lemma~\ref{lem:2-outermost}. 
Their preimages $\widetilde{M_{0}}$ and $\widetilde{M_{1}}$ in $\widetilde{X}$ 
are isotopic in $\widetilde{X}$. 
In the same manner as Theorem~\ref{thm:1-isotopy}, 
$\widetilde{M_{0}}$ and $\widetilde{M_{1}}$ are equivariantly re-embedded in the thickened torus. 
Since $M_{0}$ and $M_{1}$ are either Seifert fibered or hyperbolic, 
$M_{0}$ and $M_{1}$ are isotopic in $X$ 
by Theorem~\ref{thm:hyp-isotopy} and Lemma~\ref{lem:2-seifert-isotopy}. 

Since the remaining regions are knotted hole balls and solid tori, 
the isotopy between $M_{0}$ and $M_{1}$ 
extends to an isotopy between $X \cut L_{0}$ and $X \cut L_{1}$ 
by Theorem~\ref{thm:1-isotopy}. 
Therefore $L_{0}$ and $L_{1}$ are isotopic. 
\end{proof}

We show that cover-isotopy implies isotopy for links in the 3-torus.

\begin{thm}
\label{thm:3-isotopy}
Let $L_{0}$ and $L_{1}$ be links in $X = T^{3}$. 
Let $P \colon \widetilde{X} \to X$ be a finite covering map. 
Suppose that the lifts $\widetilde{L_{0}} = P^{-1}(L_{0})$ and $\widetilde{L_{1}} = P^{-1}(L_{1})$ are isotopic in $\widetilde{X}$. 
Then $L_{0}$ and $L_{1}$ are isotopic in $X$. 
\end{thm}

Suppose that 
the covering space $\widetilde{X}$ corresponds to 
the subgroup of $\pi_{1}(X) = \pi_{1}(\bbR^{3} / \bbZ^{3})= \bbZ^{3}$ 
generated by $\bm{a} = (a_{1}, a_{2}, a_{3})^{\mathsf{T}}$, $\bm{b} = (b_{1}, b_{2}, b_{3})^{\mathsf{T}}$, and $\bm{c} = (c_{1}, c_{2}, c_{3})^{\mathsf{T}}$. 
The degree of the covering map $P$ is $|\det(\bm{a} \ \bm{b} \ \bm{c})|$. 
The following lemma is proven in the same manner as Lemma~\ref{lem:2-seifert-isotopy}. 

\begin{lem}
\label{lem:3-seifert-isotopy}
Theorem~\ref{thm:3-isotopy} holds 
in the case that $L_{0}$ and $L_{1}$ are Seifert fibered. 
\end{lem}
\begin{proof}
Proposition~\ref{prop:3-seifert} implies that 
$L_{0}$ is isotopic to $T_{3}(p_{0}, q_{0}, r_{0})$, 
and $L_{1}$ is isotopic to $T_{3}(p_{1}, q_{1}, r_{1})$. 
We may assume that $p_{0}, p_{1} \geq 0$. 
We take $(\bm{a}, \bm{b}, \bm{c})$ as a basis for $\pi_{1}(\widetilde{X})$. 
Then $\widetilde{L_{i}}$ for $i=0,1$ is isotopic to 
$T_{3}(\tilde{p}_{i}, \tilde{q}_{i}, \tilde{r}_{i})$, 
where 
\[
\tilde{p}_{i} = \det(\bm{x}_{i} \ \bm{b} \ \bm{c}), \
\tilde{q}_{i} = \det(\bm{a} \ \bm{x}_{i} \bm{c}), \
\tilde{r}_{i} = \det(\bm{a} \ \bm{b} \ \bm{x}_{i}), \
\bm{x}_{i} = (p_{i}, q_{i}, r_{i})^{\mathsf{T}}. 
\]
Since $\widetilde{L_{0}}$ and $\widetilde{L_{1}}$ are isotopic, 
we have $p_{0} = p_{1}$, $q_{0} = q_{1}$, and $r_{0} = r_{1}$ 
by Lemma~\ref{lem:3-t-link}. 
Hence $L_{0}$ and $L_{1}$ are isotopic. 
\end{proof}

\begin{proof}[Proof of Theorem~\ref{thm:3-isotopy}]
In the same manner as Theorem~\ref{thm:1-isotopy}, 
we may assume that $L_{0}$ and $L_{1}$ are non-split links. 
The JSJ decomposition of 
$\widetilde{X} \cut \widetilde{L_{0}} \simeq \widetilde{X} \cut \widetilde{L_{1}}$ 
is lifts of those of $X \cut L_{0}$ and $X \cut L_{1}$ 
by Lemma~\ref{lem:jsj-cover}. 
Suppose that $L_{0}$ is decomposed into $L_{0,1}, \dots, L_{0,m}$ along the layering JSJ tori, 
where $L_{0,j}$ are arranged in the order of coordinates for $I$. 
Then $\widetilde{L_{0}}$ is decomposed into their lifts $\widetilde{L_{0,1}}, \dots, \widetilde{L_{0,m}}$ 
along the layering JSJ tori. 
The same applies to $L_{1}$. 
Since the decomposition along the layering JSJ tori is unique up to isotopy, 
$\widetilde{L_{0,j}}$ and $\widetilde{L_{1,j}}$ are isotopic for any $1 \leq j \leq m$ 
after possibly cyclically permuting the indices. 
Then $L_{0}$ and $L_{1}$ are isotopic by Theorem~\ref{thm:2-isotopy}. 

Suppose that $L_{0}$ and $L_{1}$ are not layered by JSJ tori. 
Theorems~\ref{thm:hyp-isotopy} and \ref{thm:1-isotopy} and Lemmas~\ref{lem:3-outermost} and \ref{lem:3-seifert-isotopy} implies that $L_{0}$ and $L_{1}$ are isotopic 
in the same manner as Theorem~\ref{thm:2-isotopy}. 
\end{proof}

We finally remark that cover-isotopy does not imply isotopy for links in the twisted $I$-bundle over the Klein bottle $S^{1} \tilde{\times} S^{1} \tilde{\times} I$. 
Three examples are shown in Figure~\ref{fig:kb_cover}. 
In each example, two non-isotopic links in $S^{1} \tilde{\times} S^{1} \tilde{\times} I$ have isotopic lifts in the double cover $T^{2} \times I$. 
In each of the first and third examples, 
the two knots in $S^{1} \tilde{\times} S^{1} \tilde{\times} I$ are distinguished by their homology classes in $H_{1}(S^{1} \tilde{\times} S^{1} \tilde{\times} I, \bbZ / 2\bbZ) \simeq \bbZ / 2\bbZ \oplus \bbZ / 2\bbZ$. 
In the second example, 
the two links in $S^{1} \tilde{\times} S^{1} \tilde{\times} I$ are distinguished by the number of components. 
In the third example, the knots are hyperbolic. 
Note that the assertion of Lemma~\ref{lem:isotopy-lift} does not hold for $X = S^{1} \tilde{\times} S^{1} \tilde{\times} I$.

\begin{figure}[ht]
\centerline{\includegraphics[width=5in]{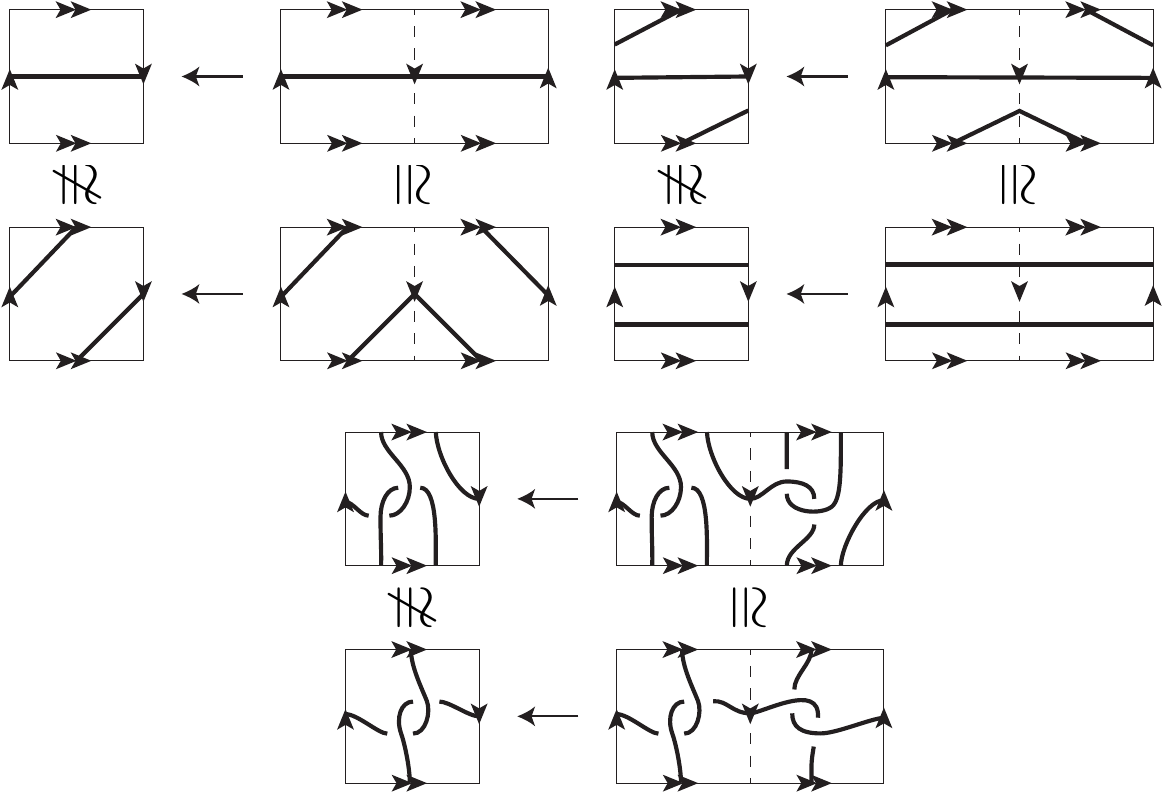}}
\vspace*{8pt}
\caption{\label{fig:kb_cover} 
Non-isotopic links in $S^{1} \tilde{\times} S^{1} \tilde{\times} I$ with isotopic lifts.}
\end{figure}

\section{Minimal motifs of periodic tangles}
\label{section:minimal}

Suppose that $X$ is $S^{1} \times D^{2}$, $T^{2} \times I$, or $T^{3}$. 
Recall the following notation. 
For links $L_{0}, L_{1} \subset X$, 
we say that $P \colon (X, L_{0}) \to (X, L_{1})$ (or $P \colon L_{0} \to L_{1}$ more simply) is a finite covering map 
if $P \colon X \to X$ is a finite covering map in the usual sense and $L_{0} = P^{-1}(L_{1})$. 
The universal cover $L_{\infty}$ of a link $L$ in $X$ is called a \emph{(singly, doubly, or triply) periodic tangle}. 
Let $L_{0, \infty}$ and $L_{1, \infty}$ be the universal covers of links $L_{0}$ and $L_{1}$ in $X$, respectively. 
The periodic tangles $L_{0, \infty}$ and $L_{1, \infty}$ are \emph{equivalent} 
if there are finite covering maps $P_{0}, P_{1} \colon X \to X$ 
such that there is an admissible homeomorphism $(X, P_{0}^{-1}(L_{0})) \to (X, P_{1}^{-1}(L_{1}))$. 
Here a homeomorphism $h \colon X \to X$ is \emph{admissible} 
if $h$ preserves the orientation in any case, 
the longitude up to isotopy in the case $X =S^{1} \times D^{2}$, 
and the boundary components (i.e. the front and back) in the case $X =T^{2} \times I$. 
We consider the equivalence classes of periodic tangles. 
A \emph{motif} of a periodic tangle $L_{\infty}$ is a link in $X$ whose lift in the universal cover of $X$ is equivalent to $L_{\infty}$. 
A motif $L$ of a periodic tangle $L_{\infty}$ is \emph{minimal} if there are no smaller motifs. 
More precisely, $L$ is minimal if the existence of a finite covering map $L \to L'$ implies the existence of an admissible homeomorphism $(X, L) \to (X, L')$. 
We show the uniqueness of a minimal motif of a periodic tangle with some conditions. 
In the case $X =T^{2} \times I$, this is equivalent to the existence of a maximally periodic diagram of a doubly periodic tangle, 
which is obtained from a diagram on $T^{2}$ of a minimal motif.

\begin{thm}
\label{thm:minimal-motif}
Let $L_{\infty}$ be a non-split periodic tangle with a motif $L \subset X = S^{1} \times D^{2}$, $T^{2} \times I$, or $T^3$. 
Suppose that each Seifert fibered JSJ piece $M$ of the complement of $L$ such that the homomorphism $\iota_{*} \colon H_{1}(M, \bbZ) \to H_{1}(X, \bbZ)$ induced by the inclusion map is non-trivial can be re-embedded as the complement of $T_{0}(p,0)$, $T_{1}(0,q)$, $T_{2}(p,q)$, or $T_{3}(p,q,r)$. 
Then $L_{\infty}$ has a motif $L_{\min}$ such that for any motif $L'$ of $L_{\infty}$, there is a finite covering map $L' \to L_{\min}$. 
In particular, $L_{\infty}$ has a minimal motif unique up to homeomorphism. 
\end{thm}

We need to consider admissible homeomorphisms for the equivalence of singly periodic tangles. 
For instance, Figure~\ref{fig6} shows two distinct links $L_{0}$ and $L_{1}$ in $S^{1} \times D^{2}$ such that there is a non-admissible homeomorphism between their double covers. 
However, $L_{0}$ and $L_{1}$ are not non-trivial covers of any other link in $S^{1} \times D^{2}$, 
since each of them has exactly one null-homotopic component.

We need the condition for splitness and JSJ pieces to avoid the following counterexamples. 
However, some periodic tangles that do not satisfy the assumption of Theorem~\ref{thm:minimal-motif} may accidentally have unique minimal motifs. 
A necessary and sufficient condition for having a unique minimal motif would be too complicated. 

There is a split doubly periodic tangle whose minimal motifs are not unique. 
For instance, the links $L_{0}$ and $L_{1}$ shown in Figure~\ref{fig:split_motif} are motifs of a common doubly periodic tangle. 
Note that a split summand (a trefoil in the figure) can be freely moved. 
However, $L_{0}$ and $L_{1}$ have no smaller motifs due to the split summand. 
The links $L_{0}$ and $L_{1}$ are distinguished by the number of components. 
We remark that splitness does not matter for singly periodic tangles. 

\begin{figure}[ht]
\centerline{\includegraphics[width=5in]{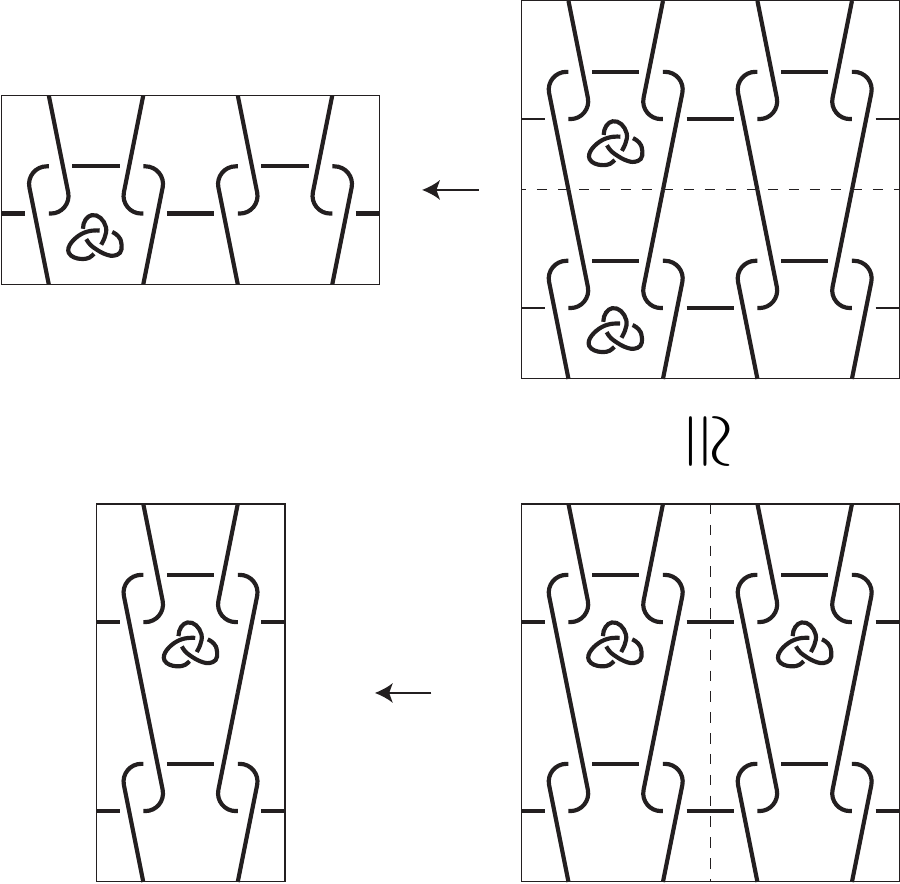}}
\vspace*{8pt}
\caption{\label{fig:split_motif} 
A split doubly periodic tangle with non-unique minimal motifs}
\end{figure}

Moreover, the links $T_{0}(3,1)$ and $T_{1}(2,1)$ in $S^{1} \times D^{2}$ are distinguished by the number of components and minimal by the classification by Lemma~\ref{lem:1-t-link} and Proposition~\ref{prop:1-seifert}. 
Since the triple cover $T_{0}(3,3)$ of $T_{0}(3,1)$ is isotopic to the double cover $T_{1}(2,2)$ of $T_{1}(2,1)$,  
the links $T_{0}(3,1)$ and $T_{1}(2,1)$ are motifs of a common singly periodic tangle. 
This corresponds to the relation $(\sigma_{1} \sigma_{2})^{3} = (\sigma_{1} \sigma_{2} \sigma_{1})^{2}$ in the braid group $B_{3}$. 
More generally, the $p$-sheeted cover $T_{0}(p,p)$ of $T_{0}(p,1)$ is isotopic to the $p-1$-sheeted cover $T_{1}(p-1,p-1)$ of $T_{1}(p-1,1)$ for $p \geq 2$. 
There does not exist such coincidence for $T_{0}(p,0)$ and $T_{1}(0,q)$. 
If a periodic tangle $L_{\infty}$ has a motif that is a connected sum of a hyperbolic link in $X$ 
and hyperbolic knots in $S^{3}$, then each Seifert fibered JSJ piece is $T_{1}(0,q)$, and so $L_{\infty}$ has a unique minimal motif by Theorem~\ref{thm:minimal-motif}. 

For links $L_{0}$ and $L_{1}$ in $X$, 
define $L_{0} \succeq L_{1}$ if there is a finite covering map $L_{0} \to L_{1}$. 
The relation $\succeq$ is a preorder on the set of links in $X$. 
Theorem~\ref{thm:minimal-motif} states that the preordered set of motifs of each periodic tangle has a least element. 
In fact, the preorder $\succeq$ induces a partial order on the set of admissible homeomorphism classes of $(X, L)$ by Lemma~\ref{lem:self-cover}. 
Similarly to Section~\ref{section:isotopy}, we use the JSJ decomposition of the complement. 
Note that there is a finite covering map $(T^{2} \times I, T_{2}(p,q)) \to (T^{2} \times I, T_{2}(p,q))$ 
of degree $n$ for any positive integer $n$.

\begin{lem}
\label{lem:self-cover}
Let $L$ be a link in $X$ whose complement is not Seifert fibered. 
Then there is an upper bound of the degrees of finite covering maps $L \to L'$ which depends only on $L$. 
In particular, there is no finite covering map $L \to L$ of degree at least two. 
\end{lem}
\begin{proof}
We define $\vol(L)$ to be the sum of the volumes of hyperbolic JSJ pieces of the complement of $L$. 
If the degree of a finite covering map $L \to L'$ is $n$, then $\vol(L) = n \vol(L')$. 
Moreover, the set of the volumes of hyperbolic 3-manifolds is well-ordered \cite{Gromov81}. 
In particular, there is a hyperbolic 3-manifold of smallest volume. 
In fact, Cao and Meyerhoff \cite{CM01} showed that 
the smallest volume of orientable cusped hyperbolic 3-manifolds is $2v_{\mathrm{tet}} = 2.02...$, 
where $v_{\mathrm{tet}}$ is the volume of a regular ideal tetrahedron. 
This volume is realized by the figure-eight knot complement. 
Consequently, if $\vol(L) \neq 0$, 
the degree of a finite covering map $L \to L'$ is at most $\vol(L) / 2v_{\mathrm{tet}}$. 
Hence we may assume that $\vol(L) = 0$. 

Suppose that $L$ maximally splits to links $L_{0}$, $L_{1}$, \dots, $L_{m}$, 
where $L_{0}$ is a (possibly empty) non-split link, 
and $L_{1}$, \dots, $L_{m}$ are contained in disjoint 3-balls. 
If $m \geq 1$, then the degree of a finite covering map from $L$ is at most $m$. 
Hence we may assume that $m=0$, i.e. $L$ is non-split. 

In the remaining case, the complement $X \cut L$ is irreducible, 
$X \cut L$ admits non-empty JSJ tori, 
and all the JSJ pieces are Seifert fibered. 
Consider the three types of JSJ tori in Lemma~\ref{lem:2-torus}. 
Each type is preserved by a finite covering map. 

Suppose that a JSJ torus bounds a knotted hole ball. 
Since a knotted hole ball is lifted to disjoint copies by a covering map, 
the degree of a finite covering map from $L$ 
is at most the number of JSJ tori bounding knotted hole balls. 

Suppose that every JSJ torus bounds a solid torus. 
Lemma~\ref{lem:jsj-cover} implies that 
a finite covering map from $L$ can be restricted to 
a finite covering map from the union of the outermost and adjacent JSJ pieces of $X \cut L$. 
Hence we may assume that every JSJ torus is a boundary component of the outermost JSJ piece. 
The outermost JSJ piece is the complement of a link $T_{2}(p,q)$. 
Consider an adjacent JSJ piece, 
which is the complement of a link $T_{l}(p', q')$ in the solid torus. 
The minimality of JSJ tori implies that $q' \neq 0$. 
The proof of Lemma~\ref{lem:1-seifert-isotopy} implies that 
the degree of a finite covering map from $(S^{1} \times D^{2}, T_{l}(p', q'))$ 
preserving the longitude 
is at most $q^{\prime}$. 
Hence the degree of a finite covering map from $L$ 
is at most $\gcd(p,q) q'$. 

Suppose that there is a layering JSJ torus of $X \cut L$. 
The degree of a finite covering map from $L$ 
is at most the degree of a finite covering map from $L_{0}$, 
where $L_{0}$ is obtained by the layering of $L$. 
Hence we may assume that $L$ is decomposed in layers into links $T_{2}(p_{0}, q_{0})$ and $T_{2}(p_{1}, q_{1})$. 
Their components are not parallel. 
The proof of Lemma~\ref{lem:2-seifert-isotopy} implies that 
the degree of a finite covering map from $L$ 
is at most $| p_{0}q_{1} - p_{1}q_{0} |$, 
which is the crossing number of $L$ by the standard projection. 

If there is a finite covering map $L \to L$ of degree at least two, 
we obtain finite covering maps $L \to L$ of arbitrarily large degree by compositing it repeatedly, 
which contradicts the assertion. 
\end{proof}

The key lemma is the following. 

\begin{lem}
\label{lem:small-motif}
Let $L_{\infty}$ be a non-split periodic tangle with a motif $L \subset X = S^{1} \times D^{2}$, $T^{2} \times I$, or $T^3$. 
Suppose that each Seifert fibered JSJ piece $M$ of the complement of $L$ such that the homomorphism $\iota_{*} \colon H_{1}(M, \bbZ) \to H_{1}(X, \bbZ)$ induced by the inclusion map is non-trivial can be re-embedded as $T_{0}(p,0)$, $T_{1}(0,q)$, $T_{2}(p,q)$, or $T_{3}(p,q,r)$. 
Let $L_{0}$ and $L_{1}$ be motifs of $L_{\infty}$. 
Then there is a motif $\widehat{L}$ of $L_{\infty}$ 
such that there are finite covering maps $L_{0} \to \widehat{L}$ and $L_{1} \to \widehat{L}$. 
\end{lem}

\begin{proof}[Proof of Theorem~\ref{thm:minimal-motif}]
If the link $T_{0}(p,0)$, $T_{1}(0,q)$, $T_{2}(p,q)$, or $T_{3}(p,q,r)$ is a motif of $L_{\infty}$, 
then $T_{0}(1,0)$, $T_{1}(0,1)$, $T_{2}(1,0)$, or $T_{3}(1,0,0)$ is a unique minimal motif of $L_{\infty}$. 
Otherwise there is a minimal motif $L_{\min}$ of $L_{\infty}$ by Lemma~\ref{lem:self-cover}. 
For any motif $L'$ of $L_{\infty}$, 
there is a motif $\widehat{L}$ such that there are finite covering maps $L_{\min} \to \widehat{L}$ and $L' \to \widehat{L}$ 
by Lemma~\ref{lem:small-motif}. 
Since $L_{\min}$ is minimal, there is a homeomorphism $\widehat{L} \to L_{\min}$. 
Hence there is a finite covering map $L' \to L_{\min}$. 
If $L'$ is also a minimal motif of $L_{\infty}$, 
then there are finite covering maps $P \colon L_{\min} \to L'$ and $Q \colon L' \to L_{\min}$. 
The covering map $Q \circ P \colon L_{\min} \to L_{\min}$ 
is a homeomorphism by Lemma~\ref{lem:self-cover}. 
Hence $P$ is also a homeomorphism. 
Therefore a minimal motif of $L_{\infty}$ is unique up to homeomorphism. 
\end{proof}

To prove Lemma~\ref{lem:small-motif}, we first consider the hyperbolic links. 

\begin{lem}
\label{lem:2-hyp-free}
Let $M$ be the complement of a hyperbolic link $L$ in $T^{2} \times I$. 
Suppose that $g \in \Isom(M)$ is extended to an admissible homeomorphism of $T^{2} \times I$. 
Moreover, suppose that the isometry $g$ acts on a component of $T^{2} \times \partial I$ 
by a non-trivial translation. 
Then $g$ has no fixed points in $M$. 
\end{lem}
\begin{proof}
Let $P \colon \widetilde{M} \to M$ denote the covering map 
that is the restriction of the universal covering map $\bbR^{2} \times I \to T^{2} \times I$. 
The isometry $g \in \Isom(M)$ lifts an isometry $\tilde{g} \in \Isom(\widetilde{M})$ 
satisfying that $P \circ \tilde{g} = g \circ P$. 
Then $\tilde{g}$ acts on a component of $\bbR^{2} \times \partial I$ by a non-trivial translation. 
Assume that $g$ has a fixed point in $M$. 
Then we can take $\tilde{g}$ that has also a fixed point $x$ in $\widetilde{M}$. 
Take a point $y$ in a horocusp neighborhood of a component of $\bbR^{2} \times \partial I$. 
Then the distances between $x = \tilde{g}^{n}(x)$ and $\tilde{g}^{n}(y)$ 
diverge to infinity when $n \to \infty$. 
This contradicts the fact that $\tilde{g}$ is an isometry. 
Hence $g$ has no fixed points in $M$. 
\end{proof}

The following lemma for $X = S^{1}\times D^{2}$ is proven in the same manner. 

\begin{lem}
\label{lem:1-hyp-free}
Let $M$ be the complement of a hyperbolic link $L$ in $S^{1} \times D^{2}$. 
Suppose that $g \in \Isom(M)$ is extended to an admissible homeomorphism of $S^{1} \times D^{2}$. 
Moreover, suppose that the isometry $g$ acts on $S^{1} \times \partial D^{2}$ 
by a non-trivial translation which does not fix a meridian $\{ \ast \} \times \partial D^{2}$ setwise. 
Then $g$ has no fixed points in $M$. 
\end{lem}

We show that Lemma~\ref{lem:small-motif} holds for the hyperbolic links. 

\begin{lem}
\label{lem:hyp-motif}
Let $\widetilde{L}, L_{0}, L_{1}$ be hyperbolic links in $X$. 
Consider the hyperbolic 3-manifold $M = X \cut \widetilde{L}$. 
Suppose that there are finite covering maps 
$P_{i} \colon \widetilde{L} \to L_{i}$ for $i=0,1$. 
If $X = S^{1} \times D^{2}$, suppose that each $P_{i}$ maps a longitude of $X$ to a longitude. 
Let $G_{i}$ denote the finite subgroups of $\Isom(M)$ consisting of the deck transformations 
for the covering maps $\widetilde{L} \to L_{i}$. 
Suppose that $G$ is the subgroup of $\Isom(M)$ generated by the elements of $G_{0} \cup G_{1}$. 
Then each non-trivial element of $G$ has no fixed points in $M$. 
Consequently, there are a link $L$ in $X$ and finite covering maps $L_{i} \to L$ 
whose deck transformation groups are $G / G_{i}$. 
\end{lem}
\begin{proof}
We first prove in the case $X = T^{2} \times I$. 
In the case $X = S^{1} \times D^{2}$, 
we can prove in the same manner using Lemma~\ref{lem:1-hyp-free}. 
Since each element of $G_{0} \cup G_{1}$ acts on a component of $T^{2} \times \partial I$ by a translation, any element of $G$ does so. 
If an element of $G$ acts trivially on a component of $T^{2} \times \partial I$, 
the element is trivial. 
Hence each non-trivial element of $G$ has no fixed points in $M$ by Lemma~\ref{lem:2-hyp-free}. 
Since the group $\Isom(M)$ is finite, the group $G$ is also finite. 
The free action of $G$ preserves each component of the boundary $T^{2} \times \partial I$ 
and maps any meridian of $\widetilde{L}$ to a meridian. 
Hence the manifold $M / G$ is the complement of a link $L$ in $T^{2} \times I$. 
Since $G_{i}$ is a subgroup of $G$, 
the covering map $M / G_{i} \to M / G$ induces a covering map $L_{i} \to L$. 

We next prove in the case $X = T^{3}$. 
The isometric action of the finite group $G$ on $M$ is extended to a smooth action of $T^{3}$. 
As in the proof of Theorem~\ref{thm:hyp-isotopy}, 
the result of Meeks and Scott \cite{MS86} implies that there is a Euclidean metric on $T^{3}$ such that the action of $G$ on $T^{3}$ is isometric. 
Since the groups $G_{0}$ and $G_{1}$ act on $T^{3}$ by translations, so does $G$. 
In particular, the action of $G$ on $T^{3}$ is free. 
Hence the manifold $M / G$ is the complement of a link $L$ in $T^{3}$. 
Since $G_{i}$ is a subgroup of $G$, 
the covering map $M / G_{i} \to M / G$ induces a covering map $L_{i} \to L$. 
\end{proof}

For finite covering maps from Seifert fibered links, 
we need the largest common quotient of them, not only the minimal motif. 
Note that the components of $T_{0}(p,q)$, $T_{2}(p,q)$, and $T_{3}(p,q,r)$ can be arbitrarily permuted by isotopy. 

\begin{lem}
\label{lem:seifert-motif}
Let $\widetilde{L}$, $L_{0}$, and $L_{1}$ be links in $X$ whose complements are $T_{0}(p,0)$, $T_{1}(0,q)$, $T_{2}(p,q)$, or $T_{3}(p,q,r)$. 
Suppose that there are finite covering maps 
$P_{i} \colon \widetilde{L} \to L_{i}$ for $i=0,1$. 
If $X = S^{1} \times D^{2}$, suppose that each $P_{i}$ maps a longitude of $X$ to a longitude. 
Let $G_{i}$ denote the deck transformation groups for the coverings $P_{i}$. 
Then there are a link $\widehat{L}$ in $X$ 
and finite covering maps 
$P'_{i} \colon \widetilde{L} \to L_{i}$ 
and $Q_{i} \colon L_{i} \to \widehat{L}$ that satisfy the following: 
\begin{enumerate}
\item The deck transformation group $G'_{i}$ of the covering map $P'_{i}$ is isomorphic to $G_{i}$. 
\item The actions of $G_{i}$ and $G'_{i}$ coincide on the boundary $\partial X$. 
\item $Q_{0} \circ P'_{0} = Q_{1} \circ P'_{1} ( = P)$. 
\item The deck transformation group $G$ for the covering $P$ is generated by the elements of $G'_{0} \cup G'_{1}$, where $G'_{0}$ and $G'_{1}$ are regarded as subgroups of $G$. 
\item If two components $K$ and $K'$ of $\widetilde{L}$ are projected to a common component of $\widehat{L}$ by $P$, then there are components $K_{0}, \dots, K_{m}$ of $\widetilde{L}$ such that $K_{j}$ and $K_{j+1}$ are projected to a common component of $L_{i}$ by $P_{i}$ for $i = 0$ or $1$. 	
\end{enumerate}
\end{lem}
\begin{proof}
We prove in the case $X = T^{2} \times I$. 
The proofs in the other cases are similar. 
The links $\widetilde{L}, L_{0}, L_{1}$ are respectively isotopic to 
$T_{2}(\tilde{p}, \tilde{q}), T_{2}(p_{0}, q_{0}), T_{2}(p_{1}, q_{1})$ 
by Proposition~\ref{prop:2-seifert}. 
Let $\tilde{d} = \gcd(\tilde{p}, \tilde{q})$, $d_{0} = \gcd(p_{0}, q_{0})$, and $d_{1} = \gcd(p_{1}, q_{1})$. 
We put $T_{2}(\tilde{p},\tilde{q})$, $T_{2}(p_{0}, q_{0})$, and $T_{2}(p_{1}, q_{1})$ 
on the standard positions of the definition in Section~\ref{section:jsj}. 
For each $i=0,1$, there is $G^{\prime}_{i}$ isomorphic to $G_{i}$ 
such that $G^{\prime}$ acts on $T^{2} \times I$ by the translations 
and the actions of $G_{i}$ and $G'_{i}$ coincide on the boundary $T^{2} \times \partial I$. 
Let $G$ be the group generated by the elements of $G'_{0} \cup G'_{1}$ as translations. 
Then $G$ is the deck transformation group for a covering map 
$P \colon (T^{2} \times I,T_{2}(\tilde{p}, \tilde{q})) \to (T^{2} \times I, T_{2}(p, q))$, 
via the covering maps $P'_{i} \colon (T^{2} \times I,T_{2}(\tilde{p}, \tilde{q})) \to (T^{2} \times I, T_{2}(p_{i}, q_{i}))$. 

We need to permute the components of $\widetilde{L}$ for Assertion (5). 
In the above case, 
the number $d = \gcd(p,q)$ of components of $\widehat{L} = T_{2}(p, q)$ satisfies $d = \gcd(d_{0}, d_{1})$. 
The orbit of each component of $\widetilde{L}$ by the action of $G$ consists of $\tilde{d} / d$ components. 
We divide the $\tilde{d}$ components of $\widetilde{L}$ maximally into the classes 
so that two components belong to a common class 
if they are projected to a common component of $L_{i}$ by $P_{i}$ for $i = 0$ or $1$. 
Two components of different classes should not be projected to a common component by $P$. 
Each class $C$ is preserved by the action of $G'_{i}$. 
Moreover, each orbit of the action of $G'_{i}$ consists of $\tilde{d} / d_{i}$ components. 
Hence the number $|C|$ of elements is a multiple of $\tilde{d} / d = \lcm(\tilde{d} / d_{0}, \tilde{d} / d_{1})$. 
We can permute the components of $\widetilde{L}$ 
so that each class is preserved by the action of $G$. 
Thus we obtain the required covering maps $P^{\prime}_{i}$. 
\end{proof}

\begin{proof}[Proof of Lemma~\ref{lem:small-motif}]
Since $L_{0}$ and $L_{1}$ are motifs of $L_{\infty}$, 
there are a link $\widetilde{L}$ in $X$ and finite covering maps $P_{i} \colon \widetilde{L} \to L_{i}$ for $i=0,1$. 
The link $\widetilde{L}$ is non-split. 
We show that there are a link $\widehat{L}$ in $X$ and finite covering maps $P'_{i} \colon \widetilde{L} \to L_{i}$ and $Q_{i} \colon L_{i} \to \widehat{L}$ that satisfy Assertions (1)--(4) in Lemma~\ref{lem:seifert-motif}. 

Consider the JSJ decomposition of $X \cut \widetilde{L}$. 
The deck transformations are restricted to the JSJ pieces. 
Moreover, they preserve the longitudes on JSJ tori that bound solid tori. 
The assertion for a hyperbolic piece follows from Lemma~\ref{lem:hyp-motif} and the Mostow rigidity theorem. 
The assertion for a Seifert fibered piece that can be re-embedded as $T_{0}(p,0)$, $T_{1}(0,q)$, $T_{2}(p,q)$, or $T_{3}(p,q,r)$ follows from Lemma~\ref{lem:seifert-motif}. 
Note that the assertion also holds in the case that the deck transformation groups act on a link component in a disjoint union of solid tori. 
By combining the actions on the JSJ pieces, 
we obtain a desired deck transformation on $X \cut \widetilde{L}$. 
Since knotted hole balls are lifted disjointly by covering maps, 
the actions are compatible on the boundary of knotted hole balls. 
If $\iota_{*} \colon H_{1}(M, \bbZ) \to H_{1}(X, \bbZ)$ is trivial for a Seifert fibered piece $M$ in $X \cut L_{i}$, then $M$ is also lifted disjointly by the covering map. 
Assertion (2) implies that the actions are compatible on a JSJ torus that is layering or bounds a solid torus. 
The compatibility for pieces adjacent to a Seifert fibered piece requires Assertion (5) in Lemma~\ref{lem:seifert-motif}. 
\end{proof}

\section*{Acknowledgements} 
The authors are grateful to Kai Ishihara, Yuya Koda, and Yuta Nozaki for their helpful discussions. 
This work was supported by the World Premier International Research Center Initiative Program, International Institute for Sustainability with Knotted Chiral Meta Matter (WPI-SKCM$^2$), MEXT, Japan, 
and JSPS Program for Forming Japan's Peak Research Universities (J-PEAKS) Grant Number JPJS00420230011. 
The first author was also supported by JSPS KAKENHI Grant-in-Aid for Early-Career Scientists Grant Number 20K14322 and Grant-in-Aid for Scientific Research (C) Grant Number 25K07005. 
The second author was also supported by JSPS KAKENHI Grant-in-Aid for Early-Career Scientists Grant Number 25K17246. 
The third author was also supported in part by National Science Foundation Grant No. DMR-1847172 and Research Corporation for the Advancement of Science Cottrell Scholar Award Grant No. CS-CSA-2020-162. 
The last author was also supported by JSPS KAKENHI Grant-in-Aid for Early-Career Scientists Grant Number 19K14530.

\bibliographystyle{plain}
\bibliography{ref-kmmy}

\end{document}